\documentclass[11pt,a4paper]{article}
\usepackage{amsfonts}
\usepackage{amsmath,amssymb}
\usepackage{mathrsfs}
\usepackage{hyperref}
\usepackage{graphicx}
\usepackage{amsthm}
\usepackage{color}

\newtheorem{theorem}{Theorem}[section]
\newtheorem{corollary}{Corollary}
\newtheorem{lemma}[theorem]{Lemma}

\newtheorem{remark}{Remark}

\numberwithin{equation}{section}\allowdisplaybreaks

\begin{document}

\title{\large\bf  Dispersive  blow-up for solutions to three dimensional generalized Zakharov-Kuznetsov equations}
\author{\normalsize \bf Yingzhe Ban$^{a}$ and \bf Minjie  Shan$^{b,*}$\\
\footnotesize
\it $^a$ Institute of Applied Physics and Computational Mathematics, Beijing 100088, P.R. China, \\
\footnotesize
\it $^b$ College of Science, Minzu University of China, Beijing 100081, P.R. China. \\  \\ \footnotesize
\it E-mails: banyingzhe22@gecaep.ac.cn, smj@muc.edu.cn   \\
} 
\date{} \maketitle
\thispagestyle{empty}

\begin{abstract}
We illustrate the dispersive blow up phenomena of the solutions of three dimensional generalized Zakharov-Kuznetsov equations. In particular, we construct smooth initial data such that, the associated global solutions fail to be $C^{1}$ at time $t$ in a null set containing all rational numbers, but are $C^{1}$ at all times $t$ which are generic irrational numbers. The key ingredient are to construct linear solutions which exhibit such phenomena and to prove nonlinear smoothing estimates for the full nonlinear model.
\end{abstract}
\section{Introduction}
\subsection{Statement of main results}
We consider the Cauchy problem for three dimensional  generalized Zakharov-Kuznetsov equations (gZK)
\begin{equation}
	\left\{
	\begin{aligned}
		&\partial_{t}u +\partial_{x}\Delta u +u^k\partial_{x} u = 0, \quad \\
		&u(0,x,y)=u_0(x,y),\ \ \ (x,y)\in \mathbb{R}\times\mathbb{R}^2, \ t\in\mathbb{R} \label{gZK} \\
	\end{aligned}
	\right.
\end{equation}
 where $k\in\mathbb{N}^+$. We will mainly focus on the case $k=1$, which is most physically relevant.
 
 Our main result is the existence of dispersive blow up solution to \eqref{gZK}. 
 \begin{theorem}\label{thm: main}
 Consider \eqref{gZK} for the case $k=1$, there exists smooth initial data $u_{0}\in C^{\infty}(\mathbb{R}^{3})$, such that the associated solution is global and one has 
 \begin{equation}
 \begin{cases}
 u(t)\in C^{1}(\mathbb{R}^{3}), \quad t>0, t\in \mathbb{R}/X,\\
 u(t)\notin C^{1}(\mathbb{R}^{3}), \quad t>0, t\in X.
 \end{cases}\nonumber
 \end{equation}
where $X$ is a measure zero set containing all rational numbers and $\mathbb{R}/X$ contains all generic irrational numbers.
 \end{theorem}

Here, we say an (irrational) number $\gamma$ is generic if it cannot be well approximated by rational numbers in the sense 
\begin{equation}\label{generic0}
|k_{1}+k_{2}\gamma|\gtrsim \frac{1}{(|k_{1}|+|k_{2}|)\ln (|k_{1}|+|k_{2}|)}, \quad \forall k_{1},k_{2}\in \mathbb{Z}. 
\end{equation}
We refer to Definition 1.1 in \cite{DGG17} for more details and we note that it is known that irrational numbers are almost surely generic irrational. One may also compare \eqref{generic0} with the Dirichlet approximation theorem.\\

One should not confuse dispersive blow up with those (dynamical) finite time blow up. Dispersive blow up, in general, means the solution of a certain dispersive equation loses regularity at certain time, but the flow can still be continued after those times. Most dispersive blow up results for nonlinear linear dispersive equations, including the one in this article, are the consequence of following two facts.
\begin{enumerate}
\item The associated linear solution exhibits dispersive blow up phenomena.
\item The Duhamel part of the nonlinear solutions exhibits nonlinear smoothing phenomena.
\end{enumerate}
The first one is of constructive nature, and the second one is obtained by refined analysis of local well-posedness.

Theorem \ref{thm: main} will essentially follows from the following two Theorems. The first one regards the construction of linear solution.
\begin{theorem}\label{thm: linear}
There exists $u_{0}\in C^{\infty}(\mathbb{R}^{3})$, such that 
\begin{equation}
\begin{cases}
W(t)u_{0}\in C^{\infty}(\mathbb{R}^{3}), \quad t>0, t\in \mathbb{R}/X,\\
W(t)u_{0}\notin C^{1}(\mathbb{R}^{3}), \quad \  t>0, t\in X.
\end{cases}\nonumber
\end{equation}
where $X$ is a measure zero set containing all rational numbers and $\mathbb{R}/X$ contains all generic irrational numbers.
\end{theorem}

See \eqref{initval000} for the explicit construction of $u_{0}$, we note that $u_{0}$ is both smooth and localized.\\

Here we use $W(t):=e^{-t\partial_{x}\Delta}$ to denote the linear propagator for the linear part of \eqref{gZK}. By Duhamel formula, the solution  to \eqref{gZK} can be written as 
\begin{equation}
u(t)=W(t)u_{0}+\int_{0}^{t}W(t-t')(u\partial_{x}u)(t')dt'.\nonumber
\end{equation}

\begin{theorem} \label{lwp}
	Let $r_1,r_2\in(0, 1)$ and $s\in (2,3)$. Assume that $u_0\in Z_{s,(r_1,r_2)}$, then there exists a positive time $T=T(\|u_0\|_{Z_{s,(r_1,r_2)}})>0$, such that \eqref{gZK} with $k=1$ has a unique solution
	$$u(t,x,y)\in C\big([0,T]; Z_{s,(r_1,r_2)})$$
which depends continuously upon  $u_0$. Furthermore, if one imposes that $u_{0}$ is small enough in $Z_{s,(r_1,r_2)}$, then one can choose $T=\infty$, i.e. the solution is global.
\end{theorem}
Here $Z_{s, (r_{1},r_{2})}$ denotes the fractional weighted Sobolev spaces by
\begin{equation}
Z_{s,(r_{1},r_{2})}=H^{s}(\mathbb{R}^{3}) \cap L^{2}((|x|^{2r_{1}}+|y|^{2r_{2}})dxdy), \nonumber
\end{equation}
where $s,r_{1},r_{2}\in \mathbb{R}$.

The second one is on  nonlinear smoothing estimate.

\begin{theorem} \label{ZKnonlinSmooth}
Let $r_1,r_2\in(0, 1)$. Assume that
 $$u_0 \in  C^{\infty}(\mathbb{R}^3) \cap \bigcap\limits_{s\in \left[2,\frac52\right)} Z_{s,(r_1,r_2)} $$ 
such that the solution  $u(t)$ of \eqref{gZK} with $k=1$ is global in time, then the Duhamel term
	\begin{equation}
		\begin{aligned}
			z_1(t):=u(t)- W(t)u_0=\int_0^t W(t-t')(u\partial_x u)(t')dt'
			\nonumber
		\end{aligned}
	\end{equation}
	is in $C^1(\mathbb{R}^3)$ for all $t>0$.
\end{theorem}

\begin{remark}
This theorem  indicates that the Duhamel term corresponding to the nonlinear equation \eqref{gZK} is  more regular than the linear term. The dispersive blow-up phenomenon is caused by singularities from the linear component part of the equation.
\end{remark}

Compared with previous results, our solutions to the 3D gZK equations, (also for the associated linear solution), display dispersive blow up at all rational time, which is dense rather than discrete in $\mathbb{R}^{+}$.  The nonlinear smoothing estimate for 3D gZK equations in the current article is also quite different compared with 2D case handled in 
\cite{2020_Linares_blowupforZK}, since gZK equations can not be symmetrized in 3D,  and associated Strichartz estimates in $y$-direction  no longer hold. We need to combine  the local smooth effect estimate and Duhamel's principle through multiple iterations. 

 Specifically, in order to show the Duhamel term lies in $H^{5/2+}(\mathbb{R}^3)$, by using the local smooth effect estimate (\ref{linear estimate:Kato-2}), we need to control the $L^1_xL^2_{yT}$ norms of $D_x^{3/2+}(uu_x)$ and $D_{y_j}^{3/2+}(uu_x)$.  By adding weight $\left<x\right>^{\frac12+}$,  we can  convert $L^1_x$ norm to  $L^2_x$ norm. Then, it can be found that the highest derivative terms will be $D^{3/2+} _xu_x$ and $D^{3/2+} _{y_j}u_x$, but the regularity for solutions is $5/2-$. For the sake of higher regularity, we expect to utilize the local smooth effect estimate again. In three dimensional case, we can not directly use Strichartz estimate  in $L^{\infty}_{xy}$-norm.  In fact, $L_T^2L^{\infty}_{xy}$ estimate holds provided that there are  derivatives in $x$-direction (see \eqref{linear estimate: strichartz cor1}). If we try to use Strichartz estimates to control $\|D_x^{5/2+}u\|_{L^r_{xyT}}$, it will be seen that we need to control $\|D_x^{2+a}(uu_x)\|_{L^2_{xyT}}$ where $0<a<1/2$. However, the worst $\|u_{xxx}\|_{L^2_TL^{\infty}_{xy}}$  is impossible to be bounded with $5/2$ regularity. Therefore, we have to adopt a distinct approach in three dimensional case. The strategy is combining  the local smooth effect estimate and Duhamel's principle through multiple iterations. This will, of course, be at the expense of controlling $\|\langle x \rangle^{1/2+}u\|_{L^2_xL^{\infty}_{yT}}$ which can be estimated by using Lemma \ref{lem:weightZK} and Lemma \ref{lem:weight-error}.

Our proof of nonlinear smoothing estimate will also extend to the general case $k\geq 2$. And those cases are indeed easier and no weight in the physical space are needed.

\begin{theorem} \label{gZKDuhamelkg2}
Let $k\geq2$ and $s=2,3,\cdots$ be given. Assume that $u(t)\in C([-T, T];H^s(\mathbb{R}^3))$ is the solution to \eqref{gZK} with initial data $u_0 \in H^s(\mathbb{R}^3)$. Denote
$$z_k(t)=\int_0^t W(t-t')(u^k\partial_x u)(t')dt',$$
then  we have
	$$z_k(t)\in C([-T, T];H^{s+1}(\mathbb{R}^3)).$$
\end{theorem}
And one can easily combine Theorem \ref{thm: linear} with Theorem \ref{gZKDuhamelkg2} to extend Theorem \ref{thm: main} to the $k\geq 2$ case.

\subsection{Background}

The Zakharov-Kuznetsov (ZK) equation (with $k=1$ in \eqref{gZK})  models the propagation of ionic-acoustic waves in magnetized plasma \cite{1974_ZK}. It is a natural multi-dimensional extension of the well-known Korteweg-de Vries  (KdV) equation. Lannes,  Linares and Saut \cite{2013_Lannes_deduce_of_zk} proved that the ZK equation is a long-wave limit of the Euler-Poisson system. Han-Kwan \cite{2013_Han_deduce_zk} derived the KdV equation  and the ZK equation from the Vlasov-Poisson equation by a long wave scaling in the cold ions limit.

The dispersive blow up for 2D ZK equation has been established in recent work of \cite{2020_Linares_blowupforZK}, where they show the existence a solution to the 2D ZK equation, which is generated by smooth initial data,  fails to be $C^{1}$ for those time $t$ equals an integer, but remains $C^{1}$ at other time. Note that their dispersive blow up, unlike ours, occur on a discrete set of time. \\
The study of dispersive blow up was initiated by \cite{1972_benjamin} for Airy equations, i.e. the linear part of KdV equations. It reflect a (re) focusing phenomena in physics, which causes strange singularity. From a mathematical view point, it means the that the regularity for the equations with smooth initial value is destroyed at some points in space-time. In some sense, this is strange because one key feature of dispersive Hamiltionian PDEs is the time translation invariance/time reversal symmetry. But it will also become clear from the construction, the (re) focusing of dispersive waves are caused by the time translational symmetry. Dispersive blow up for KdV,  gKdV, Schr\"odinger equations has been studied in \cite{1993_Bona_Dispersive_gKdV},\cite{2010_Bona_dispersive_Schroinger},\cite{2016_Hong_dispersive_NLS},\cite{1993_LinaScia}. It has been noted one can employ weighted Sobolev space to explore the nonlinear smoothing estimate which is crucial in the illustration of dispersive blow up \cite{2017_Linares}, see also \cite{2020_Linares_blowupforZK}.

There are many literatures regarding the local and global well-posedness of ZK equation, (also modified ZK equation and generalized ZK equation). For the 3D case,  we refer to \cite{2009_Linares_3dZK},\cite{2012_Ribaud_3dZK},\cite{2015_Molinet_2dZK},\cite{2016_Huo_3dZK},\cite{2014_Grunrock_3DmZK},\cite{2015_Grunrock},\cite{2017_Kato},\cite{2018_KATO} and reference therein. For the 2D case, we refer to \cite{1995_Faminskii}, \cite{2009_Linares_2dZK_mZK},  \cite{2014_Grunrock_2dZK}, \cite{2015_Molinet_2dZK}, \cite{Shan20a,Shan20b}, \cite{2021-1_KINOSHITA}, \cite{SWZ2023}, \cite{2016_Bustamante},\cite{2016_Fonseca},\cite{2012_FARAH},  \cite{2012_Ribaud_2dgZK}, \cite{2015_Grunrock},\cite{2012_Ribaud_2dgZK}, \cite{2011Linares2dgZK},\cite{2020_Bhattacharya} and reference therein.

\subsection{Notation} We give the notation that will be used throughout this paper. For $A, B \geq 0$ fixed,  $A\lesssim B$ means that $A\leq C \cdot B$ for an absolute constant $C>0$.  $A\gg B$ means that  $A>C \cdot B$ for a very large positive constant $C$.  We write $c+\equiv c+\epsilon$ and $ c-\equiv c-\epsilon$ for some  $0<\epsilon\ll 1$.

Let $1\leq p,q\leq\infty$. We define 
$$\|f\|_{L^p_xL^q_{yT}}=\left(\int_{\mathbb{R}}\Big(\int_{\mathbb{R}^2}\int^T_{-T}|f(t,x,y)|^qdt \ dy\Big)^{p/q}dx\right)^{1/p}$$
with the usual modifications if either $p=\infty$ or $q=\infty$ .  If $T=\infty$ we shall use the notation $\|f\|_{L^p_xL^q_{yt}}$. Similar definitions and considerations may be made interchanging the variables $x$, $y$ and $t$.

For $s>0$, we define $D^sf$ and $J^sf$ as
$$\widehat{D^s}f(\xi,\eta)=(\xi^2+|\eta|^2)^{s/2}\hat{f}(\xi,\eta)\hspace{4mm} and \hspace{4mm} \widehat{J^s}f(\xi,\eta)=(1+\xi^2+|\eta|^2)^{s/2}\hat{f}(\xi,\eta).$$
We denote the fractional derivatives of order $s$ with respect to $x$ and $y$ by $D^s_xf$ and $D^s_yf$
$$\widehat{D^s_x}f(\xi,\eta)=|\xi|^{s}\hat{f}(\xi,\eta)\hspace{4mm} and \hspace{4mm} \widehat{D^s_y}f(\xi,\eta)=|\eta|^{s}\hat{f}(\xi,\eta).$$
and define $J^s_xf$ and $J^s_yf$ as
$$\widehat{J^s_x}f(\xi,\eta)=(1+\xi^2)^{s/2}\hat{f}(\xi,\eta)\hspace{4mm} and \hspace{4mm} \widehat{J^s_y}f(\xi,\eta)=(1+|\eta|^2)^{s/2}\hat{f}(\xi,\eta).$$

\subsection{Organization of the paper} In Section 2, we introduce function spaces and some
estimates that will be utilized frequently later. In Section 3, we prove local well-posedness for the 3D ZK equation in weighted Sobolev spaces. Section 4 is devoted to the construction of smooth initial data such that the solution of the corresponding linear equation develops singularities at all positive rational times. In Section 5, we show that for initial data in $H^{\frac{5}{2}-}(\mathbb{R}^3)$ the Duhamel term of the 3D ZK equation is actually in $H^{\frac{5}{2}+}(\mathbb{R}^3)$. This tells us that the dispersive blow-up is from the linear part of the solution. We prove that the Duhamel term possesses higher regularity for 3D generalized ZK equations as well.

\section{Preliminaries}\label{section:linear estimate}

In this section, we introduce linear estimates for 3D ZK, including Strichartz estimate, Kato local smoothing estimate, maximal function estimate, interpolation inequality and fractional Leibniz rule.

First we state the  decay estimate  which implies  Strichartz estimate for the linear propagator of ZK equations.
\begin{lemma}[see Lemma 3.2 in \cite{2009_Linares_3dZK}]\label{lem:dispersive decay}
	Let $0<\varepsilon<1$, $\rho \in \mathbb{R}$, then
	\begin{equation*}
		I_{t}(\bar{x})=\int_{\mathbb{R}^{3}}|\xi|^{\varepsilon+i\rho} e^{i t \omega(\bar{\xi})+i \bar{x} \cdot \bar{\xi}} \mathrm{d} \bar{\xi}
	\end{equation*}
	satisfies
	\begin{equation*}
		\left\|I_{t}(\bar{x})\right\|_{L^\infty} \lesssim \frac{1}{|t|^{1+\varepsilon / 3}},
	\end{equation*}
where $\omega(\bar{\xi})=\bar{\xi}_1(\bar{\xi}_1^2+\bar{\xi}_2^2+\bar{\xi}_3^2)$.
\end{lemma}

Using the above decay estimate and Hardy-Littlewood-Paley inequality, we have
\begin{equation}\label{ineq:strichartz3}
	\big\|\int D_x^{\theta\varepsilon}W(t-t^\prime)g(\cdot,t^\prime)\mathrm{d}t^\prime\big\|_{L_t^pL^q_{xy}} \lesssim
	\|g\|_{L_t^{p'}L^{q'}_{xy}}\nonumber
\end{equation}
where $0<\varepsilon<1$,  $0<\theta<(1+\varepsilon/3)^{-1}$, $\frac{2}{p}=\theta(1+\frac{\varepsilon}{3})$ and $\frac{1}{q}=\frac{1-\theta}{2}$.
Then by TT*-argument the following Strichartz estimate is obtained.
\begin{lemma}[Strichartz estimate, see Proposition 3.1 in \cite{2009_Linares_3dZK}]\label{linear estimate:strichartz}
	Let $0<\varepsilon<1$,  $0<\theta<(1+\frac{\varepsilon}{3})^{-1}$. Then
	\begin{align}
		\left\|D_x^{\frac{\theta\varepsilon}{2}}W(t)f\right\|_{L_t^pL^q_{xy}} &\lesssim
		\|f\|_{L^2_{xy}},\label{ineq:strichartz1} \\
	\left\|\int D_x^{\frac{\theta\varepsilon}{2}}W(-t)g(\cdot,t)\mathrm{d}t\right\|_{L^2_{xy}} &\lesssim
		\|g\|_{L_t^{p'}L^{q'}_{xy}}, \label{ineq:strichartz2}
	\end{align}
where $\frac{2}{p}=\theta(1+\frac{\varepsilon}{3})$, $\frac{1}{q}=\frac{1-\theta}{2}$, $\frac{1}{p}+\frac{1}{p'}=1$ and $\frac{1}{q}+\frac{1}{q'}=1$.
\end{lemma}

\begin{corollary}\label{linear estimate: strichartz cor}
	For $0<\gamma<\beta$, we have
	\begin{equation}
		\|D^{\gamma}_x W(t)f\|_{L^2_T L^\infty_{xy}}\lesssim_T \|f\|_{H^\beta_{xy}}.\label{linear estimate: strichartz cor1}
	\end{equation}
\end{corollary}
\begin{proof}
Using Sobolev embedding theorem and Strichartz estimate \eqref{ineq:strichartz1} yields
\begin{align}
		\|D^{\gamma}_x W(t)f\|_{L^2_T L^\infty_{xy}}\lesssim_T \big\|J^{3/q}D^{\gamma-\theta\varepsilon/2}_xD^{\theta\varepsilon/2}_x W(t)f\big\|_{L^{p}_t L^{q}_{xy}}\lesssim \|f\|_{H^{3(1-\theta)/2+\gamma-\theta\varepsilon/2}_{xy}}\nonumber
	\end{align}
as long as $\gamma\geq\theta\varepsilon/2$. By taking $\varepsilon=0+$ and $\theta=1-$, one immediately gets
 \eqref{linear estimate: strichartz cor1}.
\end{proof}
\begin{remark}
	As there are no decay estimates for $y_j$ ($j=1,2$),  we can not get corresponding Strichartz estimates in $y$-direction.
\end{remark}

\begin{lemma}
	[Kato smoothing estimate]\label{Kato's smoothing effect }
	For any $f\in L^2(\mathbb{R}^3)$, $g\in L^1_x L^2_{yt}$,
\begin{align}
\big\|\nabla W(t)f \big\|_{L_x^{\infty}L^2_{yt}} &\lesssim
		\|f \|_{L^{2}_{xy}},\label{linear estimate:Kato-1} \\
		\sup \limits_{t\in[-T,T]} \big\|\nabla\int_0^t W(-t')g( t')\mathrm{d}t'\big\|_{L^2_{xy}} &\lesssim
		\|g\|_{L^1_xL^{2}_{yt}}.\label{linear estimate:Kato-2}\\
	 \big\|\nabla^2\int_0^t W(t-t')g( t')\mathrm{d}t'\big\|_{L_x^{\infty}L^2_{yT}} &\lesssim
		\|g\|_{L^1_xL^{2}_{yT}}.\label{linear estimate:Kato-3}
	\end{align}
\end{lemma}
\begin{proof}
\eqref{linear estimate:Kato-1} and \eqref{linear estimate:Kato-3} are from Proposition 3.1 and Proposition 3.6 in \cite{2012_Ribaud_3dZK}. \eqref{linear estimate:Kato-2} is the dual of \eqref{linear estimate:Kato-1}.
\end{proof}

\begin{remark}
	We can also get 
\begin{align}
\|D_x^{\frac12}D_{y_{j_1}}^{\frac12}W(t)f\|_{L^\infty_{y_{j_1}} L^2_{xy_{j_2}T}}&\lesssim \|f\|_{L^2_{xy}}\label{linear estimate:Kato-4} \\
\sup \limits_{t\in[-T,T]} \big\|\nabla\int_0^t W(-t')g(\cdot, t')\mathrm{d}t'\big\|_{L^2_{xy}} &\lesssim
		\|g\|_{L^1_{y_{j_1}} L^2_{xy_{j_2}T}}\label{linear estimate:Kato-5}
\end{align}
	where $j_1,j_2\in\{1,2\}$ and $j_1\neq j_2$.
\end{remark}

\begin{lemma}[Maximal function estimate]
	Let $0<T<1$ and $s> 1$. Assume that $f\in \mathcal{S}\left(\mathbb{R}^{3}\right)$, then 
	\begin{equation}
		\|W(t) f\|_{L_{x}^{2} L_{yT}^{\infty}} \lesssim\|f\|_{H^{s}_{xy}}. \label{linear estimate:maximal}
	\end{equation}
\end{lemma}
\begin{proof}
\eqref{linear estimate:maximal} can be found from Proposition 3.3  in \cite{2012_Ribaud_3dZK}. 
\end{proof}

We need the following interpolation inequality, fractional Leibniz rule and Weighted Kato-Ponce inequality.

\begin{lemma}[see Lemma 2.7 in \cite{2020_Linares_blowupforZK}]\label{lem:interpolaion}
	Assume that $a, b>0$, $p\in(1,\infty)$ and  $\theta \in (0,1)$. If
	 $J^a f\in L^p(\mathbb{R}^n)$  and $\left< x\right>^b f\in L^p(\mathbb{R}^n)$, then
	\begin{equation}
		\|\left< x\right>^{(1-\theta)b}J^{\theta a}f\|_{L^p}\lesssim \|\left< x\right>^b f\|_{L^p}^{1-\theta}\|J^af\|_{L^p}^{\theta}. \label{lem:interpolaion1}
	\end{equation}
The same holds for homogeneous derivatives $D^a$ in place of $J^a$. Moreover, for $p=2$,
	\begin{equation}
		\|J^{\theta a}\big(\left< x\right>^{(1-\theta)b}f\big)\|_{L^2}\lesssim \|\left< x\right>^b f\|_{L^2}^{1-\theta}\|J^af\|_{L^2}^{\theta}. \label{lem:interpolaion2}
	\end{equation}
\end{lemma}


\begin{lemma}[see Theorem 1 in \cite{KPV1993}]\label{lem:Leibniz}
	Let $s\in(0,1)$ and $p\in (1,\infty)$. Then
	\begin{equation}
		\|D^s(f g)-f D^s g-g D^s f\|_{L^p(\mathbb{R})}\lesssim \|g\|_{L^\infty(\mathbb{R})}\|D^sf\|_{L^p(\mathbb{R})}. \label{lem:Leibniz-1}
	\end{equation}
Further more, we have
	\begin{equation}
		\|D^s(f g)\|_{L^p(\mathbb{R})}\lesssim \|f D^s g\|_{L^p(\mathbb{R})} + \|g\|_{L^\infty(\mathbb{R})}\|D^s f\|_{L^p(\mathbb{R})}. \label{lem:Leibniz-12}
	\end{equation}
\end{lemma}

\begin{lemma}[see Theorem 1.1 in \cite{2016_Cruz-Uribe_KatoPonce}] \label{Kato-Ponce-weight}
	Let $1<p, q<\infty$, $\frac{1}{2}<r<\infty$ such that $\frac{1}{r}=\frac{1}{p}+\frac{1}{q}$. If $v \in A_{p}$, $w \in A_{q}$, and $s>\max \left\{0, n\left(\frac{1}{r}-1\right)\right\}$ or $s$ is non-negative even, then for any $f, g \in \mathcal{S}\left(\mathbb{R}^{n}\right)$, we have
	\begin{align}
		\left\|D^{s}(f g)-f D^{s}(g)\right\|_{L^{r}\left(v^{\frac{r}{p}} w^{\frac{r}{q}}\right)}
		\lesssim\left\|D^{s} f\right\|_{L^{p}(v)}\|g\|_{L^{q}(w)}+\|\nabla f\|_{L^{p}(v)}\left\|D^{s-1} g\right\|_{L^{q}(w)},	\label{Kato-Ponce-weight1}
		\\
		\left\|J^{s}(f g)-f J^{s}(g)\right\|_{L^{r}\left(v^{\frac{r}{p}} w^{\frac{r}{q}}\right)}
		\lesssim\left\|J^{s} f\right\|_{L^{p}(v)}\|g\|_{L^{q}(w)}+\|\nabla f\|_{L^{p}(v)}\left\|J^{s-1} g\right\|_{L^{q}(w)},\label{Kato-Ponce-weight2}
	\end{align}
	where the constants depend on $p, q, s,[v]_{A_{p}}$  and $[w]_{A_{q}}$.
\end{lemma}

\section{Well-posedness in weighted sobolev spaces}

In this section,  we focus on the local well-posedness  in weighted Sobolev space $Z_{s,(r_1,r_2)}$ for the three dimensional ZK equation.

Noting that $$e^{it\omega(\xi,\eta)}\partial_{\xi}\widehat{u_0}=\partial_\xi(e^{it\omega(\xi,\eta)}\widehat{u_0})-it(3\xi^2+\eta^2)e^{it\omega(\xi,\eta)}\widehat{u_0},$$
we get 
\begin{equation*}
	\begin{aligned}
		\quad W(t)\left(x u_{0}\right)
		&=i\mathscr{F}^{-1} \partial_{\xi}\left(e^{i t \omega(\xi,\eta )} \widehat{u}_{0}\right)+ t\mathscr{F}^{-1}\left(3 \xi^{2}+\eta^{2}\right) e^{i t \omega(\xi,\eta)} \widehat{u_{0}} \\
		&= xW(t)u_0-t W(t)\left(3 \partial_{x}^{2}+\partial_{y_1}^{2}+\partial_{y_2}^{2}\right) u_{0}.
	\end{aligned}
\end{equation*}

The above identity suggests that the regularity of the solution to ZK equation is twice the decay rate  of the solution.

However, what we used in Sobolev spaces is a fractional weight. In order to utilize all those estimates list in the previous section, we need to controll  the difference between $|x|^{r_1}W(t)u_0$ and $W(t)(|x|^{r_1}u_0)$ by $\|u_0\|_{H^s}$.

\begin{lemma}\label{lem:weightZK}
	Let $r_1,r_2\in (0, 1) $ and  $s\geqslant 2\max \{r_1,r_2\}$. Assume that  $u_0\in Z_{s,(r_1,r_2)} (\mathbb{R}^3)$, then for all $ t\in \mathbb{R}$ and almost every $(x,y_1,y_2)\in\mathbb{R}^3$, it holds that	
	\begin{equation}\label{eq:weight_x}
		|x|^{r_1}W(t)u_0=W(t)(|x|^{r_1}u_0)+W(t)\left(\left\{\Phi_{\xi,t,r_1}(\widehat{u_0})\right\}^{\vee}\right),
	\end{equation}
	\begin{equation}\label{eq:weight_y}
		|y_j|^{r_2}W(t)u_0=W(t)(|y_j|^{r_2}u_0)+W(t)\left(\left\{\Phi_{\eta_j,t,r_2}(\widehat{u_0})\right\}^{\vee}\right),
	\end{equation}
	with
\begin{align}
		\|\Phi_{\xi,t,r_1}(\widehat{u_0})\|_2&\lesssim (1+|t|)\|u_0\|_{H^s(\mathbb{R}^3)},\label{eq:error_x}\\
		\|\Phi_{\eta_j,t,r_2}(\widehat{u_0})\|_2&\lesssim (1+|t|)\|u_0\|_{H^s(\mathbb{R}^3)}, \label{eq:error_y}
	\end{align}
for $j=1,2$.
	Moreover, let $0<\beta \leq 1+\varepsilon$ for $\varepsilon\ll1$,  assume that $D^{\beta}(|x|^{r_1}u_0)$, $D^{\beta}(|y_j|^{r_2}u_0)\in L^2(\mathbb{R}^3)$ and  $u_0\in H^{s+\beta}(\mathbb{R}^3)$, then one has
	\begin{equation}\label{eq:weight_derive_x}
		D^{\beta}\left(|x|^{r_1}W(t)u_0\right)=W(t)(D^{\beta}|x|^{r_1}u_0)+W(t)\left(D^{\beta}\left\{\Phi_{\xi,t,r_1}(\widehat{u_0})\right\}^{\vee}\right),
	\end{equation}
	\begin{equation}\label{eq:weight_derive_y}
		D^{\beta}\left(|y_j|^{r_2}W(t)u_0\right)=W(t)(D^{\beta}|y_j|^{r_2}u_0)+W(t)\left(D^{\beta}\left\{\Phi_{\eta_j,t,r_2}(\widehat{u_0})\right\}^{\vee}\right),
	\end{equation}
	with
\begin{align}
		\|D^{\beta}\left\{\Phi_{\xi,t,r_1}(\widehat{u_0})\right\}^{\vee}\|_2&\lesssim (1+|t|)\|u_0\|_{H^{s+\beta}(\mathbb{R}^3)}, \label{eq:error_derive_x} \\
		\|D^{\beta}\left\{\Phi_{\eta_j,t,r_2}(\widehat{u_0})\right\}^{\vee}\|_2&\lesssim (1+|t|)\|u_0\|_{H^{s+\beta}(\mathbb{R}^3)}, 	\label{eq:error_derive_y}
	\end{align}
for $j=1,2$.
\end{lemma}

\begin{remark}
 Analogous results were established  for the $k$-generalized KdV equation \cite{FLP15} and  the 2D ZK equation  \cite{2016_Fonseca}. But, in \cite{FLP15, 2016_Fonseca}, \eqref{eq:weight_derive_x}-\eqref{eq:error_derive_y} were proved  under the condition $0<\beta< \min \{r_1,r_2\}<1$. We extend it to the case $0<\beta\leq 1+\varepsilon$ by using the following lemma,  especially \eqref{lemlem2b}.
\end{remark}

\begin{lemma}\label{lem:weight-error}
Let $\alpha \in(0,1)$ and  $p \in(1, \infty)$. Assume that $f \in L^{p}\left(\mathbb{R}^{3}\right) $ and $f \in L^{p}\left((1+x_1^{2}+x_2^{2}+x_3^2)^{\alpha p} \mathrm{d} x_1 \mathrm{d}x_2\mathrm{d}x_3 \right)$.
	For $t \in \mathbb{R}$, $\textbf{x}=\left(x_1,x_2,x_3\right) \in \mathbb{R}^{3}$,  let us denote
	\begin{equation*}
		\phi_{j, t, \alpha}(f)\left(x_{1}, x_{2}, x_3\right)=\lim _{\varepsilon \to 0} C_\alpha \int_{|s| \geqslant \varepsilon} \left(e^{i t\left(\omega(\textbf{x}+\textbf{e}_j s)-\omega(\textbf{x})\right)}-1\right)|s|^{-(1+\alpha)} f\left(\textbf{x}+\textbf{e}_j s\right) \mathrm{d} s
	\end{equation*}
 where $\omega(\textbf{x})=x_1(x_1^2+x_2^2+x_3^2)$ and $\textbf{e}_j$ denotes the unit vector in the $x_j$ direction 	for $j=1,2,3$. Then
	\begin{equation}
		\left\|\phi_{j, t, \alpha}(f)\right\|_{L^p} \leqslant c_{\alpha}(1+|t|)\left(\|f\|_{L^p}+\left\|\left(1+x_{1}^{2}+x_2^2+x_3^2\right)^{\alpha} f\right\|_{L^p}\right) \label{lemlem2a}
	\end{equation}
for $j=1,2,3$.  Moreover, if $f \in L^{p}\left((1+x_1^{2}+x_2^{2}+x_3^2)^{(\alpha+1/2) p+} \mathrm{d} x_1 \mathrm{d}x_2\mathrm{d}x_3 \right)$ and $f \in L^p_{x_kx_l}L^q_{x_j}$ where $(j,k,l)$ is a permutation of $(1,2,3)$, then we have 
\begin{align}
		&\left\|(1+x_1^{2}+x_2^{2}+x_3^2)^{1/2+}\phi_{j, t, \alpha}(f)\right\|_{L^p}  \notag \\
\leqslant &c_{\alpha}(1+|t|)\left(\left\|(1+x_{1}^{2}+x_2^2+x_3^2)^{\alpha+1/2+} f\right\|_{L^p}+\|f\|_{L^p_{x_kx_l}L^q_{x_j}}\right) \label{lemlem2b}
	\end{align}
for $j=1,2,3$, where $\frac{1}{q}>\frac{1}{p}+1-\alpha$. 
\end{lemma}

\begin{proof}
\eqref{lemlem2a} can be obtained by using a similar method provided in  \cite{FLP15}, see also Lemma 1 in \cite{2016_Fonseca}. The proof for \eqref{lemlem2b} is slightly different from that  in  \cite{FLP15}, see  \eqref{lemlem2b1}. 

Let us divide the integral in the expression of $\phi_{j, t, \alpha}(f)$ into three cases depending on whether $s$ is close or far away from the origin.

\vspace{2mm}
\noindent {\bf{Case 1. }} $|s|\geq\frac{1}{100}$

\vspace{1mm}
In this case, there is no singularity for $|s|^{-(1+\alpha)}$. Hence, we can bound
\begin{align}\left|e^{i t\left(\omega(\textbf{x}+\textbf{e}_j s)-\omega(\textbf{x})\right)}-1\right|\leq2.\label{lemlem2bbb0}\end{align}
By Minkowski's  inequality and Hardy-Littlewood-Sobolev inequality, one gets
\begin{align}
&\left\|\int_{|s| \geq 1/100} \left(e^{i t\left(\omega(\textbf{x}+\textbf{e}_j s)-\omega(\textbf{x})\right)}-1\right)|s|^{-(1+\alpha)}(1+x_1^{2}+x_2^{2}+x_3^2)^{1/2+} f(\textbf{x}+\textbf{e}_j s) \mathrm{d} s\right\|_{L^p} \notag \\
\leq& 2\left\|\int_{|s| \geq 1/100} |s|^{-(1+\alpha)}(1+x_1^{2}+x_2^{2}+x_3^2)^{1/2+} |f(\textbf{x}+\textbf{e}_j s)| \mathrm{d} s\right\|_{L^p}\notag \\
\leq& 4\int_{|s| \geq 1/100}\frac{\left\|(1+|\textbf{x}+\textbf{e}_j s|^2)^{1/2+} f(\textbf{x}+\textbf{e}_j s) \right\|_{L^p}}{|s|^{(1+\alpha)}} \mathrm{d} s+4\left\|\int\frac{|s|^{1+}|f(\textbf{x}+\textbf{e}_j s)| }{|s|^{(1+\alpha)}} \mathrm{d} s\right\|_{L^p} \notag \\
\leq& c_{\alpha} \left(\|(1+x_1^{2}+x_2^{2}+x_3^2)^{1/2+}f\|_{L^p}+\|f\|_{L^p_{x_2x_3}L^q_{x_1}}\right) \label{lemlem2b1}
\end{align}
with $\frac{1}{q}>\frac{1}{p}+1-\alpha$.

\vspace{2mm}
\noindent {\bf{Case 2. }} $|s|\leq\frac{1}{100}$ and $x_1^2+x_2^2+x_3^2<100$

\vspace{1mm}
In this case, we have the bound
\begin{align}\left|e^{i t\left(\omega(\textbf{x}+\textbf{e}_j s)-\omega(\textbf{x})\right)}-1\right|\leq |t||\omega(\textbf{x}+\textbf{e}_j s)-\omega(\textbf{x})|\leq 4(1+x_1^2+x_2^2+x_3^2)|ts|.\label{lemlem2bb0}\end{align}
Then, it follows from Minkowski’s inequality that
\begin{align}
&\left\|\int_{|s| \leq 1/100} \left(e^{i t\left(\omega(\textbf{x}+\textbf{e}_j s)-\omega(\textbf{x})\right)}-1\right)|s|^{-(1+\alpha)}(1+x_1^{2}+x_2^{2}+x_3^2)^{1/2+} f(\textbf{x}+\textbf{e}_j s) \mathrm{d} s\right\|_{L^p} \notag \\
\leq& 404|t|\left\|\int_{|s| \leq 1/100} |s|^{-\alpha}(1+x_1^{2}+x_2^{2}+x_3^2)^{1/2+} |f(\textbf{x}+\textbf{e}_j s)| \mathrm{d} s\right\|_{L^p}\notag \\
\leq& 800|t|\int_{|s| \leq 1/100}\frac{\left\|(1+|\textbf{x}+\textbf{e}_j s|^2)^{1/2+} f(\textbf{x}+\textbf{e}_j s) \right\|_{L^p}}{|s|^{\alpha}} \mathrm{d} s \notag \\
\leq& c_{\alpha} |t|\|(1+x_1^{2}+x_2^{2}+x_3^2)^{1/2+}f\|_{L^p}. \label{lemlem2b2}
\end{align}

\vspace{2mm}
\noindent {\bf{Case 3. }} $|s|\leq\frac{1}{100}$ and $x_1^2+x_2^2+x_3^2\geq100$

\vspace{1mm}
We further sub-divide this case into two sub-cases:
$$(a)  \hspace{2mm}|s|\leq \frac{1}{1+x_1^2+x_2^2+x_3^2}, \hspace{10mm}(b) \hspace{2mm} |s|\geq \frac{1}{1+x_1^2+x_2^2+x_3^2}.$$

\noindent {\bf{Case 3 (a). }} $x_1^2+x_2^2+x_3^2\geq100$  and $|s|\leq \frac{1}{1+x_1^2+x_2^2+x_3^2}<\frac{1}{100}$

\vspace{1mm}
Using \eqref{lemlem2bb0} and variable substitution $\tau=s(1+x_1^2+x_2^2+x_3^2)$, we obtain 
\begin{align}
&\left\|\int_{|s| \leq 1/100} \left(e^{i t\left(\omega(\textbf{x}+\textbf{e}_j s)-\omega(\textbf{x})\right)}-1\right)|s|^{-(1+\alpha)}(1+x_1^{2}+x_2^{2}+x_3^2)^{1/2+} f(\textbf{x}+\textbf{e}_j s) \mathrm{d} s\right\|_{L^p} \notag \\
\leq& 4|t|\left\|\int_{|\tau| \leq 1 }(1+x_1^{2}+x_2^{2}+x_3^2)^{\alpha+1/2+}|\tau|^{-\alpha}|f(\textbf{x}+\textbf{e}_j \tau(1+x_1^2+x_2^2+x_3^2)^{-1})| \mathrm{d} \tau\right\|_{L^p}\notag \\
\leq& 8|t|\int_{|\tau| \leq 1}\frac{\left\|(1+|\textbf{x}+\textbf{e}_j s|^2)^{\alpha+1/2+} f(\textbf{x}+\textbf{e}_j s) \right\|_{L^p}}{|\tau|^{\alpha}} \mathrm{d} \tau \notag \\
\leq& c_{\alpha} |t|\|(1+x_1^{2}+x_2^{2}+x_3^2)^{\alpha+1/2+}f\|_{L^p}. \label{lemlem2b3}
\end{align}

\noindent {\bf{Case 3 (b). }} $ \frac{1}{1+x_1^2+x_2^2+x_3^2}\leq |s|\leq\frac{1}{100}$

\vspace{1mm}
Using \eqref{lemlem2bbb0} and the same variable substitution $\tau=s(1+x_1^2+x_2^2+x_3^2)$ as above, one gets 
\begin{align}
&\left\|\int_{|s| \leq 1/100} \left(e^{i t\left(\omega(\textbf{x}+\textbf{e}_j s)-\omega(\textbf{x})\right)}-1\right)|s|^{-(1+\alpha)}(1+x_1^{2}+x_2^{2}+x_3^2)^{1/2+} f(\textbf{x}+\textbf{e}_j s) \mathrm{d} s\right\|_{L^p} \notag \\
\leq& 4 \left\|\int_{1\leq|\tau| \leq\frac{1+x_1^{2}+x_2^{2}+x_3^2}{100} }(1+x_1^{2}+x_2^{2}+x_3^2)^{\alpha+1/2+}|\tau|^{-1-\alpha}|f(\textbf{x}+\textbf{e}_js)| \mathrm{d} \tau\right\|_{L^p}\notag \\
\leq& 8 \int_{1\leq|\tau|}\frac{\left\|(1+|\textbf{x}+\textbf{e}_j s|^2)^{\alpha+1/2+} f(\textbf{x}+\textbf{e}_j s) \right\|_{L^p}}{|\tau|^{1+\alpha}} \mathrm{d} \tau \notag \\
\leq& c_{\alpha} \|(1+x_1^{2}+x_2^{2}+x_3^2)^{\alpha+1/2+}f\|_{L^p}. \label{lemlem2b4}
\end{align}

Then, by collecting \eqref{lemlem2b1}-\eqref{lemlem2b4}, we get the desired estimate \eqref{lemlem2b}. 
\end{proof}

Now we apply this lemma to show Lemma \ref{lem:weightZK}.

\vspace{1mm}
{\bf{Proof of Lemma \ref{lem:weightZK}.}} By the definition of Stein derivation \cite{Stein_HarmonicAnalysis}, one has
	\begin{align}
		& |x|^{r_1}W(t)u_0-W(t)(|x|^{r_1}u_0) \notag \\
		=&\mathscr{F}^{-1}\left(D_\xi^{r_1}(e^{i t\omega(\xi,\eta)} \widehat{u_0})-e^{i t\omega(\xi,\eta)} D_\xi^{r_1} \widehat{u_0}\right)\notag  \\
		=&\mathscr{F}^{-1}\lim_{\varepsilon\to 0}C_{r_1}\int_{|s|\geq\varepsilon}
		\left(e^{it\omega(\xi+s,\eta)}\widehat{u_0}(\xi+s,\eta)-e^{it\omega(\xi,\eta)}\widehat{u_0}(\xi+s,\eta)\right)|s|^{-(1+r_1)}\mathrm{d}s\notag  \\
		=&C_{r_1}\mathscr{F}^{-1}e^{it\omega(\xi,\eta)}\lim_{\varepsilon\to 0}\int_{|s|\geq\varepsilon}
		\left(e^{it\left(\omega(\xi+s,\eta)-\omega(\xi,\eta)\right)}-1\right)|s|^{-(1+r_1)}\widehat{u_0}(\xi+s,\eta)\mathrm{d}s\notag  \\
=&W(t)\left(\left\{\Phi_{\xi,t,r_1}(\widehat{u_0})\right\}^{\vee}\right),\label{WPhi}
	\end{align}
where
$$\Phi_{\xi, t, r_{1}}\left(\widehat{u_0}\right)=C_{r_1}\lim_{\varepsilon\to 0}\int_{|s|\geq\varepsilon}
		\left(e^{it\left(\omega(\xi+s,\eta)-\omega(\xi,\eta)\right)}-1\right)|s|^{-(1+r_1)}\widehat{u_0}(\xi+s,\eta)\mathrm{d}s.$$

Moreover, for $\beta \in\left(0, r_{1}\right)$, it is easy to verify that
	\begin{align}
		&\quad D_{x}^{\beta}\left(|x|^{r_{1}} W(t) u_{0}\right)- W(t)\left(D_{x}^{\beta}|x|^{r_{1}} u_{0}\right) \notag \\
		&=\mathscr{F}^{-1}|\xi|^{\beta}D_{\xi}^{r_1}e^{it\omega(\xi,\eta)}\widehat{u_0}-\mathscr{F}^{-1}e^{it\omega(\xi,\eta)}|\xi|^{\beta}D_{\xi}^{r_1}\widehat{u_0}  \notag\\
		&=\mathscr{F}^{-1}|\xi|^{\beta}e^{it\omega(\xi,\eta)}\Phi_{\xi,t,r_1}  \notag\\
		&=W(t)\left(D^{\beta}\left\{\Phi_{\xi,t,r_1}(\widehat{u_0})\right\}^{\vee}\right). \label{DWPhi}
	\end{align}
From \eqref{WPhi} and \eqref{DWPhi}, we see that \eqref{eq:weight_x} and \eqref{eq:weight_derive_x} hold true. 

Next, we show \eqref{eq:error_x} and \eqref{eq:error_derive_x} respectively. It follows from Lemma \ref{lem:weight-error} that  
	\begin{equation}
		\left\|\Phi_{\xi, t, r_{1}}\left(\widehat{u}_{0}\right)\right\|_{2}  \lesssim (1+|t|)\left(\left\|\widehat{u}_{0}\right\|_{2}+\left\|\left(1+\xi^{2}+\eta^{2}\right)^{r_{1}} \widehat{u}_{0}\right\|_{2}\right) \lesssim (1+|t|)\|u_0\|_{H^s(\mathbb{R}^3)} \nonumber
	\end{equation}
	provided $2r_1<s$.
	
By Parseval's identity and the proof of \eqref{lemlem2b},  we may then deduce that
	\begin{align}	
\|D^{\beta}\left\{\Phi_{\xi,t,r_1}(\widehat{u_0})\right\}^{\vee}\|_2&=\left\||\xi|^{\beta}\Phi_{\xi,t,r_1}(\widehat{u_0})\right\|_2 \notag \\
&\lesssim (1+|t|)\left(\|(1+\xi^2+\eta^2)^{r_1+\beta/2}\widehat{u_0}\|_{2}+\|\widehat{u_0}\|_{2}\right)\notag \\
&\lesssim (1+|t|)\|u_0\|_{H^{s+\beta}}\nonumber
	\end{align}
for $0<\beta \leq 1+\varepsilon$ and $s\geq2r_1$. 

Other estimates can be verified parallelly. Thus we finish the proof. \hspace{19mm}$\square$

Let us turn to the well-posedness for 3D ZK in weighted spaces. We define the work space as
$$X_T=\Big\{u\in C([0,T];Z_{s,(r_1,r_2)}):\|u\|_{X_T} <\infty \Big\}$$
where
	\begin{align}
		\|u\|_{X_T}= &
		\|u\|_{L^{\infty}_TH^s_{xy}}+\big\||x|^{r_1}u\big\|_{L^{\infty}_TL^2_{xy}}+\big\||y|^{r_2}u\big\|_{L^{\infty}_TL^2_{xy}}+\|u_x\|_{L^{2}_{T}L^{\infty}_{xy}}+\|u_{xx}\|_{L^{2}_{T}L^{\infty}_{xy}}\notag \\
&+\|D^{s-1}_{x}u\|_{L_T^{2}L^{\infty}_{xy}}+\sum_{j=1}^{2}\|D^{s-2}_{y_j}u_x\|_{L_T^{2}L^{\infty}_{xy}}+\sum_{j=1}^{2}\|u_{xy_j}\|_{L^{2}_{T}L^{\infty}_{xy}}+\|u\|_{L^{2}_xL^{\infty}_{yT}}\notag \\
		&+\|u_x\|_{L^{2}_xL^{\infty}_{yT}}+\sum_{j=1}^{2}\|u_{y_j}\|_{L^{2}_xL^{\infty}_{yT}}+\|D^{s-2}_{x}u\|_{L^{2}_xL^{\infty}_{yT}}+\sum_{j=1}^{2}\|D^{s-2}_{y_j}u\|_{L^{2}_xL^{\infty}_{yT}}\notag \\
&+\|D_x^su_x\|_{L^{\infty}_xL^{2}_{yT}}
		+\sum_{j=1}^{2}\|D_{y_j}^su_x\|_{L^{\infty}_xL^{2}_{yT}}+\|u_{xxx}\|_{L^{\infty}_xL^{2}_{yT}}+\sum_{j=1}^{2}\|D_{y_j}^2u_x\|_{L^{\infty}_xL^{2}_{yT}}\nonumber
	\end{align}
with $2<s<3$.

Local well-posedness for 3D ZK equation in weighted Sobolev spaces $Z_{s,(r_1,r_2)}$ can be accomplished by using the contraction
principle. In fact, we will show that the operator $$\mathscr{T} :u \mapsto W(t)u_0-\int^t_0W(t-t')(u\partial_xu)(t')\mathrm{d}t'.$$
is a contraction mapping on $X_T$ for some $T>0$.

\begin{lemma}\label{lwplem}
Let $2<s<3$. For $u\in X_T$, let us denote
$$ z_1(t)=\int^t_0 W(t-t')(u\partial_x u)(t')\mathrm{d}t'.$$
Then, we have
\begin{align}
\|z_1(t)\|_{L^\infty_T H^s_{xy}}\lesssim T^{1/2}\|u\|_{X_T}^2 .\label{lwplem0}
	\end{align}
\end{lemma}

\begin{proof}
	We only estimate  $\|D^s_{y_1}z_1\|_{L^\infty_T L^2_{xy}}$, as other terms can be handled in a similar way.
	
By Leibniz's rule for fractional derivatives, it follows that
\begin{align}
&\|D^s_{y_1}z_1(t)\|_{L^\infty_T L^2_{xy}}
\leq \int^T_0  \|D^s_{y_1}(uu_x)\|_{L^2_{xy}} \mathrm{d}t \notag \\
\lesssim \ & \int^T_0  \left(\|D^{s-2}_{y_1}(u_{y_1y_1}u_x)\|_{L^2_{xy}}+\|D^{s-2}_{y_1}(u_{y_1}u_{xy_1})\|_{L^2_{xy}}+\|D^{s-2}_{y_1}(u_{xy_1y_1}u)\|_{L^2_{xy}} \right)\mathrm{d}t \notag \\
\lesssim \ &  T^{1/2}\Big(\|D^{s}_{y_1}u\|_{L_T^{\infty}L^2_{xy}}\|u_x\|_{L_T^{2}L^{\infty}_{xy}}+\|u_{y_1y_1}\|_{L_T^{\infty}L^2_{xy}}\|D^{s-2}_{y_1}u_x\|_{L_T^{2}L^{\infty}_{xy}}\notag \\
&  +  \|D^{s-1}_{y_1}u_x\|_{L_x^{\infty}L^2_{yT}}\|u_{y_1}\|_{L_x^{2}L^{\infty}_{yT}}+\|D^{s-1}_{y_1}u\|_{L_T^{\infty}L^2_{xy}}\|u_{xy_1}\|_{L_T^{2}L^{\infty}_{xy}}\notag \\
&  + \|D^{s}_{y_1}u_x\|_{L_x^{\infty}L^2_{yT}}\|u\|_{L_x^{2}L^{\infty}_{yT}}+\|u_{xy_1y_1}\|_{L_x^{\infty}L^2_{yT}}\|D^{s-2}_{y_1}u\|_{L_x^{2}L^{\infty}_{yT}}\Big)\notag \\
\lesssim&  T^{1/2}\Big(\|u\|_{L_T^{\infty}H^s_{xy}}\|u_x\|_{L_T^{2}L^{\infty}_{xy}}+\|u\|_{L_T^{\infty}H^s_{xy}}\|D^{s-2}_{y_1}u_x\|_{L_T^{2}L^{\infty}_{xy}}\notag \\
&  +  \|u\|_{L_T^{\infty}H^s_{xy}}\|u_{y_1}\|_{L_x^{2}L^{\infty}_{yT}}+\|u\|_{L_T^{\infty}H^s_{xy}}\|u_{xy_1}\|_{L_T^{2}L^{\infty}_{xy}}\notag \\
&  + \|D^{s}_{y_1}u_x\|_{L_x^{\infty}L^2_{yT}}\|u\|_{L_x^{2}L^{\infty}_{yT}}+\|D^{2}_{y_1}u_{x}\|_{L_x^{\infty}L^2_{yT}}\|D^{s-2}_{y_1}u\|_{L_x^{2}L^{\infty}_{yT}}\Big)\notag \\
\lesssim \ &  T^{1/2}\|u\|_{X_T}^2 \nonumber
	\end{align}
which completes the proof of \eqref{lwplem0}. 
\end{proof}

Now we concentrate on the local well-posedness for the 3D ZK equation.

\textbf{Proof of Theorem \ref{lwp}. } The original equation \eqref{gZK} can be rewritten as an integral equation
$$u= \chi_T(t)W(t)u_0-\chi_T(t)\int^t_0W(t-t')(u\partial_xu)(t')\mathrm{d}t':=\mathscr{T} u ,$$
where $\chi_T$ is a smooth cut-off function such that $\chi_T$ is nonnegative and $\chi_T=1$ on $[0,T]$.

\vspace{1mm}

Next, we estimate  respectively $\mathscr{T} u$ in every norm involved in $X_T$.

\vspace{1mm}
\noindent {\bf{(i)  Estimate for $ \|\mathscr{T} u\|_{H^s_{xy}}$.}}

It follows from  Lemma \ref{lwplem} that
\begin{align}
	\|\mathscr{T} u\|_{L^{\infty}_{T}H^s_{xy}}\leq  \|u_0\|_{H^s(\mathbb{R}^3)}+\|z(t)\|_{L^{\infty}_{T}H^s_{xy}}\lesssim  \|u_0\|_{H^s(\mathbb{R}^3)}+T^{1/2}\|u\|^2_{X^T}.\label{Hs0}
\end{align}

\noindent {\bf{(ii)  Estimate for $\||x|^{r_1}\mathscr{T} u\|_{L^2_{x y}}$.}}

Applying  Minkovski's inequality, Lemma \ref{lem:weightZK} and Lemma \ref{lwplem}, we have
\begin{align}
	&\quad \big\||x|^{r_1}\mathscr{T} u\big\|_{L^2_{xy}}
	\lesssim  \big\||x|^{r_1}W(t)u_0\big\|_{L^2_{xy}}+\int_0^T\big\||x|^{r_1}W(t-t')(uu_x)\big\|_{L^2_{xy}}\mathrm{d}t'\notag\\
	&\lesssim  \big\||x|^{r_1}u_0\big\|_{L^2_{xy}}+(1+T)\|u_0\|_{H^s}+\int_0^T\big\||x|^{r_1}uu_x\big\|_{L^2_{xy}}\mathrm{d}t'+(1+T)\int_0^T\|uu_x\|_{H^s}\mathrm{d}t' ,\notag\\
&\lesssim  \big\||x|^{r_1}u_0\big\|_{L^2_{xy}}+(1+T)\|u_0\|_{H^s}+T^{\frac{1}{2}}\big\||x|^{r_1}u\big\|_{L_T^{\infty}L^2_{xy}}\|u_x\|_{L_T^{2}L^{\infty}_{xy}}+(1+T)T^{\frac{1}{2}}\|u\|_{X_T}^2 ,\notag\\
&\lesssim (1+T)\|u_0\|_{Z_{s, (r_1, r_2)}}+(1+T)T^{\frac{1}{2}}\|u\|_{X_T}^2. \label{xHs1}
\end{align}

\noindent {\bf{(iii)  Estimate for $\|D^{s-2}_{y_j}\partial_x\mathscr{T} u\|_{L_T^{2}L^{\infty}_{xy}}$.}}

\vspace{1mm}
We take $\|D^{s-2}_{y_j}\partial_x\mathscr{T} u\|_{L_T^{2}L^{\infty}_{xy}}$ for instance, as other terms 
$$\|(\mathscr{T}u)_x\|_{L^{2}_{T}L^{\infty}_{xy}},\hspace{2mm}\|(\mathscr{T}u)_{xx}\|_{L^{2}_{T}L^{\infty}_{xy}}
,\hspace{2mm}\|D^{s-1}_{x}\mathscr{T}\|_{L_T^{2}L^{\infty}_{xy}},\hspace{2mm}\sum_{j=1}^{2}\|(\mathscr{T}u)_{xy_j}\|_{L^{2}_{T}L^{\infty}_{xy}}
$$
can be controlled in a similar way.

From Corollary \ref{linear estimate: strichartz cor} and Lemma \ref{lwplem}, we obtain 
\begin{align}
		&\quad \|D^{s-2}_{y_j}\partial_x\mathscr{T} u\|_{L_T^{2}L^{\infty}_{xy}} \notag\\
		&\lesssim_{T} \|u_0\|_{H^{s}(\mathbb{R}^3)}+\int_0^T\|uu_x\|_{H^{s}(\mathbb{R}^3)}\mathrm{d}t' \notag\\
		&\lesssim _{T}\|u_0\|_{H^s(\mathbb{R}^3)}+T^{1/2}\| u\|^2_{X_T}.\label{xHs2}
\end{align}

\noindent {\bf{(iv)  Estimate for $\|D^{s-2}_{y_j}\mathscr{T} u\|_{L^{2}_xL^{\infty}_{yT}}$.}}

Using Minkovski's inequality, the maximal function estimate  \eqref{linear estimate:maximal} and Lemma \ref{lwplem}, we see that
\begin{align}
	\|D^{s-2}_{y_j}\mathscr{T} u\|_{L^{2}_xL^{\infty}_{yT}}&\lesssim   \|u_0\|_{H^{s}(\mathbb{R}^3)}+\int_0^T\|uu_x\|_{H^{s}(\mathbb{R}^3)}\mathrm{d}t\notag\\
	&\lesssim \|u_0\|_{H^s(\mathbb{R}^3)}+T^{1/2}\|u\|_{X_T}^2.\label{xHs3}
\end{align}
Other terms 
$$\|\mathscr{T}u\|_{L^{2}_xL^{\infty}_{yT}},\hspace{2mm}\|(\mathscr{T}u)_x\|_{L^{2}_xL^{\infty}_{yT}},\hspace{2mm}\sum_{j=1}^{2}\|(\mathscr{T}u)_{y_j}\|_{L^{2}_xL^{\infty}_{yT}},\hspace{2mm}\|D^{s-2}_{x}\mathscr{T}u\|_{L^{2}_xL^{\infty}_{yT}}$$
can be controlled in the same manner.

\noindent {\bf{(v)  Estimate for $\|D_{y_j}^s(\mathscr{T} u)_x\|_{L^{\infty}_xL^{2}_{yT}}$.}}

Kato smoothing effect \eqref{linear estimate:Kato-1} and Lemma \ref{lwplem} yield that
\begin{align}
	\|D_{y_j}^s(\mathscr{T} u)_x\|_{L^{\infty}_{x}L^2_{yT}}&\lesssim  \|u_0\|_{H^s(\mathbb{R}^3)}+\int_0^T\|D_{y_j}^s(uu_x)\|_{L^{2}_{xy}}\mathrm{d}t \notag\\
	&\lesssim \|u_0\|_{H^s(\mathbb{R}^3)}+T^{1/2}\|u\|_{X_T}^2.\label{xHs4}
\end{align}
One can similarly control
$$ \|D_x^s(\mathscr{T} u)_x\|_{L^{\infty}_xL^{2}_{yT}},
		\hspace{2mm}\|(\mathscr{T} u)_{xxx}\|_{L^{\infty}_xL^{2}_{yT}},\hspace{2mm}
\|D_{y_j}^2(\mathscr{T} u)_x\|_{L^{\infty}_xL^{2}_{yT}}.$$

Collecting estimates \eqref{Hs0}-\eqref{xHs4}, we obtain
\begin{align}
	\|\mathscr{T} u\|_{X_T}&< C_1(1+T)^{\alpha} \left(\|u_0\|_{Z_{s,(r_1, r_2)}}+T^{1/2}\|u\|_{X_T}^2\right).
\end{align}
where positive constants $C_1$ and $\alpha$. 

Taking $r=4C_1\|u_0\|_{Z_{s,(r_1, r_2)}}$, it is easy to see that  $\mathscr{T}$ is a contraction mapping on 
$$ B_r=\left\{u\in X_T\ \big| \hspace{2mm} \|u\|_{X_T} <r\right\}$$  for $T=\min\left\{1,  (4C_1r)^{-2}\right\}$. Thus, there exists a unique solution $u$ to \eqref{gZK} satisfying
\begin{align}
	\| u\|_{X_T} \leqslant 4C_1\|u_0\|_{Z_{s,(r_1, r_2)}}.\nonumber
\end{align}
This completes the proof of Theorem \ref{lwp}.\hspace{77mm}$\square$

\section{Linear singularities}\label{Linearsingu}

We study dispersive singularities for solution  to 3D ZK in this section. It will be shown that  the solution to 3D ZK displays dispersive blow up for some smooth initial data and the blow up appears due to the linear component of the solution.

Specifically, we will construct an initial data  $u_0\in C^\infty(\mathbb{R}^3)\cap Z_{s,(r_1,r_2)}(\mathbb{R}^3) $ for $s=5/2-$ such that the free solution part $W(t)u_0$ fails to be in $C^1(\mathbb{R}^3)$ at all positive rational time  $t\in \mathbb{Q}^+$. However, the Duhamel term $z_1(t)=\int^t_0 W(t-t')(u\partial_x u)(t')\mathrm{d}t'$ holds good regularity at any time $t\in \mathbb{R}$, as to be showed in the next section that $z_1(t)$ lies in $H^{\frac{5}{2}+}(\mathbb{R}^3)$ and therefore embedded in $C^1	(\mathbb{R}^3)$.

 Denote 
$$\varphi(x, y_1, y_2)=e^{-2\sqrt{x^2+y_1^2+y_2^2}}$$
which is in $Z_{s,(r_1,r_2)}(\mathbb{R}^3)$ for any $s\in\left[2, 5/2\right)$. Besides, it is easy to verify that 
$$e^{x+y_1+y_2}\varphi \in L^2(\mathbb{R}^3) \hspace{2mm} \text{and} \hspace{2mm} \varphi \in C^{\infty}(\mathbb{R}^3\backslash(0,0,0))\backslash C^1(\mathbb{R}^3).$$
Considering that the regularity of $\varphi $ to be $C^1$ is broken at the origin, we will construct the initial value $u_0$ based on $\varphi$ such that $W(t)u_0$ will display dispersive singularities at positive rational time.

\begin{lemma}  \label{linbuplema}
Let $\varphi=e^{-2\sqrt{x^2+y_1^2+y_2^2}}$. Then,  
$W(t)\varphi\in C^{\infty}(\mathbb{R}^3) $ for $t\neq0$.
Moreover, for $t>0$ we have
\begin{align}
\left\|\partial^{\alpha}_{x,y} e^{x+y_1+y_2}W(t)\varphi\right\|_{L^2_{xy}}
		\lesssim   t^{-\frac{|\alpha|}{2}} e^{3t} \left\|e^{x+y_1+y_2}\varphi\right\|_{L^2}  
		<  \infty, \label{linbuplema1}
	\end{align}
and for $t<0$ we have 
\begin{align}
\left\|\partial^{\alpha}_{x,y} e^{-x-y_1-y_2}W(t)\varphi\right\|_{L^2_{xy}}
		\lesssim   |t|^{-\frac{|\alpha|}{2}} e^{-3t} \left\|e^{-x-y_1-y_2}\varphi\right\|_{L^2}  
		<  \infty,  \label{linbuplema2}
	\end{align}
where $\alpha=(\alpha_0, \alpha_1,\alpha_2)$, $\partial^{\alpha}_{x,y}=\partial^{\alpha_0}_{x}\partial^{\alpha_1}_{y_1}\partial^{\alpha_2}_{y_2}$ and  $|\alpha|=\alpha_0+\alpha_1+\alpha_2$.
\end{lemma}
\begin{proof} In fact, from Sobolev embedding theorem it follows that 
\begin{align}
W(t)\varphi\in C^{\infty}(\mathbb{R}^3) 
\Longleftrightarrow& \hspace{1mm} e^{x+y_1+y_2}W(t)\varphi\in C^{\infty}(\mathbb{R}^3)	\notag \\
\Longleftrightarrow&  \hspace{1mm} e^{x+y_1+y_2}W(t)\varphi\in H^{m}(\mathbb{R}^3) \hspace{3mm} \text{for all}  \hspace{2mm} m\in \mathbb{N}.	\nonumber	
\end{align}
Hence, we only need to show \eqref{linbuplema1}.

 Set $v(t)=W(t)\varphi$. Then $v(t)$ is the solution to the following linear equation
\begin{align} 
	\partial_t v+\partial_x\Delta v=0 \label{linearZK}
\end{align}
with initial data $v|_{t=0}=\varphi$. 

Let us denote  $w(t)=e^{x+y_1+y_2}v(t)$. Putting $v(t)=e^{-x-y_1-y_2}w(t)$  into  \eqref{linearZK}
gives the equation that $w(t)$ satisfies
\begin{equation*}
	\partial_t w+ \partial_x\Delta w-\Delta w-2\partial_x(\partial_x+\partial_{u_1}+\partial_{y_2})w +5\partial_xw +2\partial_{y_1}w +2\partial_{y_2}w-3w=0
\end{equation*}
with initial data $w|_{t=0}=w_0=e^{x+y_1+y_2}\varphi$.
Hence,
\begin{equation*}
	\widehat{w}=e^{t\big(i\xi(\xi^2+\eta^2)-(\xi^2+\eta^2)-2\xi(\xi+\eta_1+\eta_2)-i(5\xi_1+2\xi_2+2\xi_3)+3\big)}\widehat{w_0}.
\end{equation*}
By Plancheral's identity, one gets
\begin{equation*}
	\begin{aligned}
		&\left\|\partial^{\alpha}_{x,y} w\right\|_{L^2}
		= \big\||\xi|^{\alpha_0}|\eta_1|^{\alpha_1}|\eta_2|^{\alpha_2}\widehat{w}\big\|_{L^2}\\
		= & \left\||\xi|^{\alpha_0}|\eta_1|^{\alpha_1}|\eta_2|^{\alpha_2}e^{-t\big(\xi^2+(\eta_1+\xi)^2+(\eta_2+\xi)^2\big)}e^{3t}\widehat{w_0}\right\|_{L^2}\\
		\lesssim &e^{3t}\big\||\xi|^{\alpha_0}|\eta_1|^{\alpha_1}|\eta_2|^{\alpha_2}e^{-t(\xi^2+(\eta_1+\xi)^2+(\eta_2+\xi)^2)}\big\|_{L^{\infty}}
	\left\|e^{x+y_1+y_2}\varphi\right\|_{L^2}\\
		\lesssim  & |t|^{-\frac{|\alpha|}{2}} 	e^{3t}\left\|e^{x+y_1+y_2}\varphi\right\|_{L^2}
		< \infty.
	\end{aligned} 	  	
\end{equation*}

We finish the proof of the lemma.
\end{proof}
Next we constructe a smooth initial data $u_0$ such that it inherits the regularity of $W(t)\varphi$ for $t\neq0$, but as time $t$ translates forward
 to $\mathbb{Q}^+$ the linear solution to \eqref{gZK} with  initial data $u_0$ will display singularity of $\varphi$ at origin. 
\begin{theorem}  \label{linbupthm}
Assume that
\begin{align}
	u_0=\sum_{k\in\mathbb{Z}^+} \sum_{\substack{j\in\mathbb{Z}^+,\\ gcd(j,k)=1}}e^{-e^{k}}e^{-j^2} W\big(-\frac{j}{k}\big)\varphi  \label{initval000}
\end{align}
where $\varphi=e^{-2\sqrt{x^2+y_1^2+y_2^2}}$ , then  we have 
\begin{equation*}
	\left\{\begin{array}{ll}
		W(t)u_0\in C^{\infty}	(\mathbb{R}^3),    &t>0, \ t\in \mathbb{R}\setminus X,\\
		W(t)u_0\in C^{\infty}	(\mathbb{R}^3\backslash(0,0,0))\backslash C^1(\mathbb{R}^3),  \ \  &t>0, \ t\in \mathbb{Q}\subset X.
	\end{array}\right.	
\end{equation*}
\end{theorem}
\begin{proof}
It follows from  \eqref{linbuplema2} and \eqref{initval000} that 
	\begin{align}
		&\left\|\partial^{\alpha}_{x,y} e^{-(x+y_1+y_2)}u_0\right\|_{L^2}\notag\\
		\lesssim & \sum_{k\in\mathbb{Z}^+} \sum_{\substack{j\in\mathbb{Z}^+,\\ gcd(j,k)=1}}e^{-e^{k}}e^{-j^2} \left\|\partial^{\alpha}_{x,y} e^{-(x+y_1+y_2)}W\big(-\frac{j}{k}\big)\varphi\right\|_{L^2} \notag\\
		\lesssim & \sum_{k\in\mathbb{Z}^+} \sum_{\substack{j\in\mathbb{Z}^+,\\ gcd(j,k)=1}}e^{-e^{k}}e^{-j^2} j^{-\frac{|\alpha|}{2}}k^{\frac{|\alpha|}{2}} e^{3j/k} \left\|e^{-x-y_1-y_2}\varphi\right\|_{L^2}  <\infty \nonumber
	\end{align}
which deduces that
$u_0\in C^{\infty}	(\mathbb{R}^3)$
by Sobolev embedding theorem.

Note that
\begin{align}
W(t)	u_0&=\sum_{k\in\mathbb{Z}^+} \sum_{\substack{j\in\mathbb{Z}^+,\\ gcd(j,k)=1}}e^{-e^{k}}e^{-j^2} W\big(t-\frac{j}{k}\big)\varphi,\notag\\
&=\sum_{j/k<t}e^{-e^{k}}e^{-j^2} W\big(t-\frac{j}{k}\big)\varphi+\sum_{j/k>t}e^{-e^{k}}e^{-j^2} W\big(t-\frac{j}{k}\big)\varphi,\notag\\
&:=W_1(t)	u_0+W_2(t)	u_0.\nonumber
	\end{align}
For any $t\in \mathbb{R}\setminus X$ and $t>0$ given, from \eqref{generic0} we have
$$|t-j/k|\gtrsim (|j|+|k|)^{-3},$$
so
	\begin{align}
		&\left\|\partial^{\alpha}_{x,y} e^{x+y_1+y_2}W_1(t)	u_0\right\|_{L^2}\notag\\
		\lesssim &\mathop{\sum\sum}_{\substack{k\in\mathbb{Z}^+, \ j\in\mathbb{Z}^+, \\ gcd(j,k)=1, \ j/k<t}}e^{-e^{k}}e^{-j^2} \left\|\partial^{\alpha}_{x,y} e^{x+y_1+y_2}W\big(t-\frac{j}{k}\big)\varphi\right\|_{L^2} \notag\\
		\lesssim&\mathop{\sum\sum}_{\substack{k\in\mathbb{Z}^+, \ j\in\mathbb{Z}^+, \\ gcd(j,k)=1, \ j/k<t}}e^{-e^{k}}e^{-j^2} (t-j/k)^{-\frac{|\alpha|}{2}} e^{3(t-j/k)} \left\|e^{x+y_1+y_2}\varphi\right\|_{L^2}  \notag\\
\lesssim& e^{3t} \left\|e^{x+y_1+y_2}\varphi\right\|_{L^2} \mathop{\sum\sum}_{\substack{k\in\mathbb{Z}^+, \ j\in\mathbb{Z}^+, \\ gcd(j,k)=1, \ j/k<t}}e^{-e^{k}}e^{-j^2} (|j|+|k|)^{\frac{3|\alpha|}{2}}  <\infty \nonumber
	\end{align}
and 
	\begin{align}
		&\left\|\partial^{\alpha}_{x,y} e^{-x-y_1-y_2}W_2(t)	u_0\right\|_{L^2}\notag\\
		\lesssim &\mathop{\sum\sum}_{\substack{k\in\mathbb{Z}^+, \ j\in\mathbb{Z}^+, \\ gcd(j,k)=1, \ j/k>t}}e^{-e^{k}}e^{-j^2} \left\|\partial^{\alpha}_{x,y} e^{-x-y_1-y_2}W\big(t-\frac{j}{k}\big)\varphi\right\|_{L^2} \notag\\
		\lesssim&\mathop{\sum\sum}_{\substack{k\in\mathbb{Z}^+, \ j\in\mathbb{Z}^+, \\ gcd(j,k)=1, \ j/k>t}}e^{-e^{k}}e^{-j^2} (j/k-t)^{-\frac{|\alpha|}{2}} e^{3(j/k-t)} \left\|e^{-x-y_1-y_2}\varphi\right\|_{L^2}  <\infty \nonumber
	\end{align}
by using Lemma  \ref{linbuplema}.
Thus  $W(t)u_0\in C^{\infty}(\mathbb{R}^3)$ for $t\in \mathbb{R}\setminus X$ and $t>0$.

However,  for $t=j_0/k_0\in\mathbb{Q}^+$,
	\begin{align}
	W\big(\frac{j_0}{k_0}\big)u_0=\mathop{\sum\sum}_{(j,k)\neq(j_0,k_0)} e^{-e^{k}}e^{-j^2} W\big(\frac{j_0}{k_0}-\frac{j}{k}\big)\varphi+ e^{-e^{k_0}}e^{-j_0^2} \varphi.	 \label{rationaltime} 	
\end{align}
It is easy to see that the last term of \eqref{rationaltime} fails to be $C^1(\mathbb{R}^3)$ while the other terms are in $C^\infty(\mathbb{R}^3)$. This completes the proof of the theorem.
\end{proof}


\section{Nonlinear smoothing}
This section devotes to show Theorem \ref{ZKnonlinSmooth} and Theorem \ref{gZKDuhamelkg2}.

\subsection{Nonlinear smoothing for $k=1$}
\textbf{Proof of Theorem \ref{ZKnonlinSmooth}. }
We will show that the Duhamel term defined as
 $$ z_1(t)=\int^t_0 W(t-t')(u\partial_x u)(t')\mathrm{d}t'$$ 
belongs to $H^{5/2+}(\mathbb{R}^3)$
for initial value 
\begin{align}
u_0\in \bigcap\limits_{s\in \left[2,\frac52\right)} Z_{s,(r_1,r_2)}\nonumber
\end{align}
where $0<r_1,r_2<1$. This  implies in particular $z_1(t)\in C^1(\mathbb{R}^3)$ by Sobolev embedding theorem. Combining with the conclusion in Section \ref{Linearsingu}, we may deduce that dispersive blow-up inherits from the linear component part of \eqref{gZK} for the initial data constructed in \eqref{initval000}.

Let us start by considering $\big\|D_x^{\frac{5}{2}+}z_1\big\|_{L^2_{xy}}$. Applying the dual version of the smoothing effect \eqref{linear estimate:Kato-2} and Cauchy-Schwarz inequality, we have
\begin{align}
	&\left\|D_x^{\frac{5}{2}+}\int^t_0 W(t-t')(u\partial_x u)(t')\mathrm{d}t'\right\|_{L^2_{xy}}\notag \\
	\lesssim&  \big\|D_x^{\frac{3}{2}+}(u\partial_x u)\big\|_{L^1_xL^2_{yT}}
\lesssim  \big\|\left<x\right>^{\frac{1}{2}+}D_x^{\frac{3}{2}+}(u\partial_x u)\big\|_{L^2_{xyT}}\notag \\
	\lesssim& \big\|\left<x\right>^{\frac{1}{2}+}[D_x^{\frac{3}{2}+},u]\partial_x u\big\|_{L^2_{xyT}}+\big\|\left<x\right>^{\frac{1}{2}+}uD_x^{\frac{3}{2}+}u_x\big\|_{L^2_{xyT}}
:= NL_1+NL_2.	\nonumber
\end{align}

\noindent
$\bullet$ {\bf Estimate for  $NL_1$.}  
\vspace{1mm}

It follows from the weighted Kato-Ponce inequality \eqref{Kato-Ponce-weight1} that
\begin{equation*}
  \begin{aligned}
	&\big\|\langle x\rangle^{\frac{1}{2}+}[D_x^{\frac32+},u]u_x\big\|_{L^2_{x}}\\
	  \lesssim & \big\| \langle x\rangle^{\frac{1}{4}+} D_x^{\frac32+}u\big\|_{L^4_x} \big\|\langle x\rangle^{\frac{1}{4}+} u_x\big\|_{L^4_{x}}
+\big\|\langle x\rangle^{\frac{1}{4}+}D_x^{\frac12+}u_x\big\|_{L^4_x}\big\|\langle x\rangle^{\frac{1}{4}+} u_x\big\|_{L^4_{x}}
   \end{aligned}
\end{equation*}
which yields  by H{\"o}lder's inequality 
  \begin{align}
	NL_1&=\big\|\langle x\rangle^{\frac{1}{2}+}[D_x^{\frac32+},u]u_x\big\|_{L^2_{xyT}}\notag \\
	&\lesssim  \left\| \big\|\langle x\rangle^{\frac{1}{4}+} D_x^{\frac32+}u\big\|_{L^4_x}\big\|\langle x\rangle^{\frac{1}{4}+} u_x\big\|_{L^4_{x}}\right\|_{L^2_{yT}} \notag \\
	&\lesssim \|\langle x\rangle^{\frac{1}{4}+} D_x^{\frac32+}u\|_{L^{p_0}_{T}L^4_{xy}} \|\langle x\rangle^{\frac{1}{4}+} u_x\|_{L^{\frac{2p_0}{p_0-2}}_{T}L^4_{xy}} \label{NL1a}
   \end{align}
where $2<p_0<4$ (to be determined later) such that $(4,p_0)$ is a Strichartz pair.

On one hand, using interpolation inequality  \eqref{lem:interpolaion1}, we have
\begin{equation*}
	\begin{aligned}
		\big\|\langle x\rangle^{\frac{1}{4}+} D_x^{\frac32+}u\big\|_{L^4_{x}}
		\lesssim  \big\|\langle x\rangle^{\frac{1}{4\beta}+} u\big\|_{L^4_{x}}^\beta 
\big\|D_x^{3/2(1-\beta)+}u\big\|_{L^4_{x}}^{1-\beta}
		\lesssim  \big\|\langle x\rangle^{\frac{1}{4\beta}+} u\big\|_{L^4_{x}}+
\big\|D_x^{3/2(1-\beta)+}u\big\|_{L^4_{x}}
	\end{aligned}
\end{equation*}
which gives
	\begin{align}
		\big\|\langle x\rangle^{\frac{1}{4}+} D_x^{\frac32+}u\big\|_{L^{p_0}_{T}L^4_{xy}}
		\lesssim \big\|\langle x\rangle^{\frac{1}{4\beta}+} u\big\|_{L^{p_0}_{T}L^4_{xy}} + \big\|D_x^{3/2(1-\beta)+}u\big\|_{L^{p_0}_{T}L^4_{xy}}.\label{NL1a1aa}
	\end{align}

For $\|\langle x\rangle^{\frac{1}{4\beta}+} u\|_{L^{p_0}_{T}L^4_{xy}}$, using Sobolev embedding theorem and interpolation inequality, 
	\begin{align}
		\big\|\langle x\rangle^{\frac{1}{4\beta}+} u\big\|_{L^4_{xy}}
		&\lesssim \big\|J^{\frac34}\langle (x,y)\rangle^{\frac{1}{4\beta}+} u\big\|_{L^2_{xy}} \notag\\
		&\lesssim \big\|J^{\frac{3}{4(1-\sigma)}} u\big\|_{L^2_{xy}}^{1-\sigma} \big\|\langle (x,y)\rangle^{\frac{1}{4\beta\sigma}+} u\big\|_{L^2_{xy}}^{\sigma}\notag\\
		&\lesssim \big\|J^{\frac{3}{4(1-\sigma)}} u\big\|_{L^2_{xy}} +\big\|\langle (x,y)\rangle^{\frac{1}{4\beta\sigma}+} u\big\|_{L^2_{xy}}.\label{NL1a1a}
	\end{align}
Taking $\beta=\frac{5} {14}+$, $\sigma=\frac{7} {10}-$ such that  $\frac{3}{4(1-\sigma)}<\frac52$  and  $\frac{1}{4\beta\sigma}<1-$, then from \eqref{NL1a1a} we have
	\begin{align}
		\big\|\langle x\rangle^{\frac{1}{4\beta}+} u\big\|_{L^{p_0}_{T}L^4_{xy}}
		\lesssim  T^{\frac{1}{p_0}}\left(\big\|J^{\frac{3}{4(1-\sigma)}} u\big\|_{L^\infty_T L^2_{xy}} +\big\|\langle (x,y)\rangle^{\frac{1}{4\beta\sigma}+} u\big\|_{L^\infty_T L^2_{xy}}\right)\lesssim \|u\|_{X_T}. \label{NL1a1b}
	\end{align}

For $\|D_x^{3/2(1-\beta)+}u\|_{L^{p_0}_{T}L^4_{xy}}$, by using Strichartz estimates \eqref{ineq:strichartz1}  with $\theta=\frac12$,  $\varepsilon=1-$ and  $p_0=\frac{4}{1+\varepsilon/3}$ and Lemma \ref{lwplem},  we get
\begin{align}
	&\big\|D_x^{3/2(1-\beta)+}u\big\|_{L^{p_0}_{T}L^4_{xy}}=\big\|D_x^{\frac{7}{3}+}u\big\|_{L^{p_0}_{T}L^4_{xy}} \notag\\
	\lesssim& 	\big\|D_x^{\frac{7}{3}+}W(t)u_0\big\|_{L^{p_0}_{T}L^4_{xy}}+\left\|D_x^{\frac{7}{3}+}\int_0^T W(t-\tau)(uu_x)\mathrm{d}\tau\right\|_{L^{p_0}_{T}L^4_{xy}}\notag\\
	\lesssim & \big\|D_x^{\frac{25}{12}+} u_0\big\|_{L^2_{xy}}+\int_0^T\big\|D_x^{\frac{25}{12}+} (uu_x)\big\|_{L^2_{xy}}dt\notag\\
 	\lesssim  &\| u_0\|_{H^{5/2-}}+\|u\|_{X_T}^2.\label{NL1a1c}
\end{align}

Thus, \eqref{NL1a1b}, \eqref{NL1a1c} together with \eqref{NL1a1aa} yield 
	\begin{align}
		\big\|\langle x\rangle^{\frac{1}{4}+} D_x^{\frac32+}u\big\|_{L^{p_0}_{T}L^4_{xy}}
		\lesssim \| u_0\|_{H^{5/2-}}+\|u\|_{X_T}+\|u\|_{X_T}^2<\infty.\label{NL1a1A}
	\end{align}

On the other hand, using interpolation inequality  \eqref{lem:interpolaion1} again, we have
	\begin{align}
	\|\langle x\rangle^{\frac{1}{4}+} u_x\|_{L^{\frac{2p_0}{p_0-2}}_{T}L^4_{xy}} 
		\lesssim \big\|\langle x\rangle^{\frac{1}{4\beta}+} u\big\|_{L^{\frac{2p_0}{p_0-2}}_{T}L^4_{xy}} + \big\|D_x^{1/(1-\beta)}u\big\|_{L^{\frac{2p_0}{p_0-2}}_{T}L^4_{xy}}.\label{NL1bbbb1bb}
	\end{align}
It is easy to see that the first term on the right-hand side of \eqref{NL1bbbb1bb} can be controlled by $\|u\|_{X_T}$ as processed in \eqref{NL1a1b}. Moreover,
	\begin{align}
\big\|D_x^{1/(1-\beta)}u\big\|_{L^{\frac{2p_0}{p_0-2}}_{T}L^4_{xy}}=\big\|D_x^{14/9+}u\big\|_{L^{6-}_{T}L^4_{xy}}
\lesssim T^{\frac{1}{6}+}\big\|J^{3/4}D_x^{14/9+}u\big\|_{L^{\infty}_{T}L^2_{xy}}\lesssim \|u\|_{X_T}\nonumber
	\end{align}
which implies immediately that 
	\begin{align}
	\|\langle x\rangle^{\frac{1}{4}+} u_x\|_{L^{\frac{2p_0}{p_0-2}}_{T}L^4_{xy}} 
		\lesssim \|u\|_{X_T}.\label{NL1bbbb1bb0}
	\end{align}

Combining  \eqref{NL1a1A}, \eqref{NL1bbbb1bb0} and  \eqref{NL1a}, we obtain
$$NL_1\lesssim \|u\|_{X_T}\left(\| u_0\|_{H^{5/2-}}+\|u\|_{X_T}+\|u\|_{X_T}^2\right)<\infty.$$

\noindent
$\bullet$ {\bf Estimate for  $NL_2$.}  
\vspace{1mm}

By H{\"o}lder's inequality, we see that
	\begin{align}
		NL_2\leqslant \big\|\langle x\rangle^{\frac{1}{2}+}u\big\|_{L^{2}_{x}L^{\infty}_{yT}}\big\|D_x^{\frac32+}u_x\big\|_{L^{\infty}_{x}L^2_{yT}}. \label{NL2a}
	\end{align}

Let us  consider the first term on the right-hand side of \eqref{NL2a}. By applying Duhamel principle, one gets
$$\left<x\right>^{\frac{1}{2}+}u=\left<x\right>^{\frac{1}{2}+}W(t)u_0-\left<x\right>^{\frac{1}{2}+}\int^t_0 W(t-t')(u\partial_x u)(t')\mathrm{d}t'.$$
Then, using the identity \eqref{eq:weight_x}, \eqref{lemlem2b} and the maximal function estimate \eqref{linear estimate:maximal}, we deduce that
	\begin{align}
	\big\|\langle x\rangle^{\frac{1}{2}+}u\big\|_{L^{2}_{x}L^{\infty}_{yT}}&\lesssim 
	\big\|\langle x\rangle^{\frac{1}{2}+}W(t)u_0\big\|_{L^{2}_{x}L^{\infty}_{yT}}+\left\|\left<x\right>^{\frac{1}{2}+}\int^t_0 W(t-t')(u\partial_x u)(t')\mathrm{d}t'\right\|_{L^{2}_{x}L^{\infty}_{yT}} \notag\\
&\lesssim \big\|J^{1+}\langle x\rangle^{\frac{1}{2}+}u_0\big\|_{L^{\infty}_{T}L^{2}_{xy}}+\int^T_0\big\|J^{1+}\langle x\rangle^{\frac{1}{2}+}uu_x\big\|_{L^{2}_{xy}}dt+\|u_0\|_{H^{2+}_{xy}}\notag\\
&\hspace{5mm} +\|u_0\|_{L^{2}_{y}L^{1+}_x}+\int^T_0\|uu_x\|_{H^{2+}_{xy}}dt+\int^T_0\|uu_x\|_{L^{2}_{y}L^{1+}_x}dt. \label{NL2a1}
	\end{align}

From interpolation inequality  \eqref{lem:interpolaion2}, it follows that
\begin{align}
\big\|J^{1+}\langle x\rangle^{\frac{1}{2}+}u_0\big\|_{L^{2}_{xy}}\lesssim  \big\|\langle x\rangle^{\frac{3}{2}+}u_0\big\|^{1/3}_{L^{2}_{xy}}\big\|J^{3/2+}u_0\big\|^{2/3}_{L^{2}_{xy}}\lesssim \|u_0\|_{Z_{s,(r_1, r_2)}}\label{NL2a1a1}
	\end{align}
and 
\begin{align}
\int^T_0\big\|J^{1+}\langle x\rangle^{\frac{1}{2}+}uu_x\big\|_{L^{2}_{xy}}dt&\lesssim  \int^T_0\big\|\langle x\rangle^{\frac{3}{2}+}uu_x\big\|^{1/3}_{L^{2}_{xy}}\big\|J^{3/2+}uu_x\big\|^{2/3}_{L^{2}_{xy}}dt \notag\\
&\lesssim  \int^T_0\big\|\langle x\rangle^{\frac{3}{2}+}uu_x\big\|_{L^{2}_{xy}}dt+ \int^T_0\big\|J^{3/2+}uu_x\big\|_{L^{2}_{xy}}dt \notag\\ 
&\lesssim  T^{1/2}\big\|\langle x\rangle^{\frac{3}{2}+}u\big\|_{L^{\infty}_{T}L^{2}_{xy}} \|u_x\|_{L^{2}_{T}L^{\infty}_{xy}}+\|u\|^2_{X_T}  \notag\\ 
&\lesssim\|u\|^2_{X_T} \label{NL2a1a2}
	\end{align}
for $s=5/2-$ and $r_1=r_2=1-$.
Moreover,
\begin{align}
\|u_0\|_{L^{2}_{y}L^{1+}_x}\lesssim \big\|\langle x\rangle^{\frac{1}{2}+}u_0\big\|_{L^{2}_{xy}}\lesssim 
\|u_0\|_{Z_{s,(r_1, r_2)}}
\label{NL2a1a3}
	\end{align}
and 
\begin{align}
\int^T_0\|uu_x\|_{L^{2}_{y}L^{1+}_x}dt\lesssim \int^T_0\big\|\langle x\rangle^{\frac{1}{2}+}uu_x\big\|_{L^{2}_{xy}}dt\lesssim\|u\|^2_{X_T}.
\label{NL2a1a4}
	\end{align}

Then, from \eqref{NL2a1}-\eqref{NL2a1a4}, we get
	\begin{align}
	\big\|\langle x\rangle^{\frac{1}{2}+}u\big\|_{L^{2}_{x}L^{\infty}_{yT}}\lesssim 
\|u_0\|_{Z_{s,(r_1, r_2)}}+	\|u\|^2_{X_T}. \label{NL2A1}
	\end{align}

In the next place, we consider the second term on the right-hand side of \eqref{NL2a}. From Duhamel's principle and local smoothing estimate \eqref{linear estimate:Kato-3}, one easily sees that
\begin{align}
\big\|D_x^{\frac32+}u_x\big\|_{L^{\infty}_{x}L^2_{yT}}\lesssim& \big\|D_x^{\frac52+}W(t)u_0\big\|_{L^{\infty}_{x}L^2_{yT}}
+\left\|D_x^{\frac52+}\int^t_0 W(t-t')(u\partial_x u)(t')\mathrm{d}t'\right\|_{L^{\infty}_{x}L^2_{yT}} \notag\\ 
\lesssim&\|u_0\|_{H^{\frac{3}{2}+}}+\big\|D_x^{\frac{1}{2}+}(uu_x)\big\|_{L^{1}_{x}L^2_{yT}}.
 \nonumber
	\end{align}
Note that
\begin{align}
	& \big\|D_x^{\frac{1}{2}+}(uu_x)\big\|_{L^1_xL^2_{yT}}
\lesssim  \big\|\left<x\right>^{\frac{1}{2}+}D_x^{\frac{1}{2}+}(uu_x)\big\|_{L^2_{xyT}}\notag \\
	\lesssim& \big\|\left<x\right>^{\frac{1}{2}+}[D_x^{\frac{1}{2}+},u]u_x \big\|_{L^2_{xyT}}+\big\|\left<x\right>^{\frac{1}{2}+}uD_x^{\frac{1}{2}+}u_x\big\|_{L^2_{xyT}}
:= \widetilde{NL}_1+\widetilde{NL}_2.	\nonumber
\end{align}
Proceeding in the same manner as above, we obtain
$$\widetilde{NL}_1\lesssim \|u\|_{X_T}\left(\| u_0\|_{H^{5/2-}}+\|u\|_{X_T}+\|u\|_{X_T}^2\right)<\infty$$
and 
\begin{align}
		\widetilde{NL}_2\lesssim \big\|\langle x\rangle^{\frac{1}{2}+}u\big\|_{L^{\infty}_{T}L^{2}_{xy}}\big\|D_x^{\frac32+}u\big\|_{L^{2}_{T}L^{\infty}_{xy}}\lesssim\|u\|_{X_T}^2. \nonumber
	\end{align}
These two estimates  further yield
\begin{align}
	\big\|D_x^{\frac32+}u_x\big\|_{L^{\infty}_{x}L^2_{yT}}\lesssim\| u_0\|_{H^{5/2-}}+
\|u\|_{X_T}\left(\| u_0\|_{H^{5/2-}}+\|u\|_{X_T}+\|u\|_{X_T}^2\right)<\infty. \label{NL2aNL2a}
\end{align}
 By \eqref{NL2A1} and  \eqref{NL2aNL2a}, we deduce that $NL_2$ can be controlled. Consequently, we obtain the required bound for $\big\|D^{\frac52+}_x z_1\big\|_{L^2_{xy}}$.

Now we consider $\big\|D_{y_j}^{\frac{5}{2}+}z_1\big\|_{L^2_{xy}}$.  In view of  the dual version of the smoothing effect \eqref{linear estimate:Kato-2} and the  fractional Leibniz’s rule \eqref{lem:Leibniz-12}, one gets
\begin{align}
	&\left\|D_{y_j}^{\frac{5}{2}+}\int^t_0 W(t-t')(u\partial_x u)(t')\mathrm{d}t'\right\|_{L^2_{xy}}\notag \\
	\lesssim&  \big\|D_{y_j}^{\frac{3}{2}+}(u\partial_x u)\big\|_{L^1_xL^2_{yT}}\notag \\
\lesssim&  \big\|\left<x\right>^{\frac{1}{2}+}D_{y_j}^{\frac{3}{2}+}(u\partial_x u)\big\|_{L^2_{xyT}}\notag \\
\lesssim & \big\|\left<x\right>^{\frac{1}{2}+}D_{y_j}^{\frac{1}{2}+}(u_{y_j}u_x )\big\|_{L^2_{xyT}}
+\big\|\left<x\right>^{\frac{1}{2}+}D_{y_j}^{\frac{1}{2}+}(u_{xy_j}u )\big\|_{L^2_{xyT}}
\notag \\
	\lesssim& \big\|D_{y_j}^{\frac{3}{2}+}u \big\|_{L_T^{\infty}L^2_{xy}}\big\|\left<x\right>^{\frac{1}{2}+}u_x \big\|_{L_T^2L^{\infty}_{xy}}+\big\|D_{y_j}^{\frac{1}{2}+}u_x \big\|_{L_T^2L^{\infty}_{xy}}\big\|\left<x\right>^{\frac{1}{2}+}u_{y_j} \big\|_{L^{\infty}_{T}L_{xy}^2}\notag \\
&\hspace{1mm}+\big\|\left<x\right>^{\frac{1}{2}+}D_{y_j}^{\frac{1}{2}+}u \big\|_{L_T^{\infty}L^2_{xy}}\|u_{xy_j} \|_{L_T^2L^{\infty}_{xy}}+\big\|\left<x\right>^{\frac{1}{2}+}uD_{y_j}^{\frac{3}{2}+}u_x\big\|_{L^2_{xyT}}.\label{nonlinddyyy}
\end{align}

It is easy to see that 
$$\big\|D_{y_j}^{\frac{3}{2}+}u \big\|_{L_T^{\infty}L^2_{xy}}, \hspace{2mm} \big\|D_{y_j}^{\frac{1}{2}+}u_x \big\|_{L_T^2L^{\infty}_{xy}}, \hspace{2mm}\text{and} \hspace{2mm}\|u_{xy_j} \|_{L_T^2L^{\infty}_{xy}}$$
 can be controlled by $\|u\|_{X_T}$ with $5/2-$ regularity from the local theory. In addition,
$$\big\|\left<x\right>^{\frac{1}{2}+}D_{y_j}^{\frac{1}{2}+}u \big\|_{L_T^{\infty}L^2_{xy}}  \hspace{4mm}\text{and} \hspace{4mm} \big\|\left<x\right>^{\frac{1}{2}+}u_{y_j} \big\|_{L^{\infty}_{T}L_{xy}^2}$$
are bounded via using interpolation inequality  \eqref{lem:interpolaion1} (one can refer to  \eqref{NL2a1a1} for more details). And, we  can  control 
$$\big\|\left<x\right>^{\frac{1}{2}+}uD_{y_j}^{\frac{3}{2}+}u_x\big\|_{L^2_{xyT}}$$
in the same manner as \eqref{NL2a}.

All that is left is to control $\big\|\left<x\right>^{\frac{1}{2}+}u_x \big\|_{L_T^2L^{\infty}_{xy}}$.  To this end, noting that 
\begin{align}\big\|\left<x\right>^{\frac{1}{2}+}u_x \big\|_{L_T^2L^{\infty}_{xy}}&\leq \big\|\partial_x\big(\left<x\right>^{\frac{1}{2}+}u\big)\big\|_{L_T^2L^{\infty}_{xy}}+\big\|\left<x\right>^{-\frac{1}{2}+}u \big\|_{L_T^2L^{\infty}_{xy}}\notag \\
&\leq \big\|\partial_x\big(\left<x\right>^{\frac{1}{2}+}u\big)\big\|_{L_T^2L^{\infty}_{xy}}+\big\|u \big\|_{L_T^2L^{\infty}_{xy}}\notag \\
&\leq \big\|\partial_x\big(\left<x\right>^{\frac{1}{2}+}u\big)\big\|_{L_T^2L^{\infty}_{xy}}+T^{1/2}\|u \|_{L_T^{\infty}H^{5/2-}_{xy}},\nonumber
\end{align}
we only need to estimate $ \big\|\partial_x\big(\left<x\right>^{\frac{1}{2}+}u\big)\big\|_{L_T^2L^{\infty}_{xy}}$.
By applying Duhamel's principle, Lemma \ref{lem:weightZK}, Lemma \ref{lem:weight-error} and Corollary \ref{linear estimate: strichartz cor}, we get
\begin{align}
&\big\|\partial_x\big(\left<x\right>^{\frac{1}{2}+}u\big)\big\|_{L_T^2L^{\infty}_{xy}}\notag \\
\lesssim&
\big\|\partial_x\big(\left<x\right>^{\frac{1}{2}+}W(t)u_0\big)\big\|_{L_T^2L^{\infty}_{xy}}+\int_0^T\big\|\partial_x\left<x\right>^{\frac{1}{2}+}W(t-t')(uu_x) \big\|_{L_T^2L^{\infty}_{xy}}dt'\notag \\
\lesssim& \big\|J^{1+}\left<x\right>^{\frac{1}{2}+}u_{0} \big\|_{L^{2}_{xy}}+\int_0^T\big\|J^{0+}\left<x\right>^{\frac{1}{2}+}\partial_x(uu_x) \big\|_{L^{2}_{xy}}dt+ \|u_0\|_{H^{2+}_{xy}} \notag \\
&\hspace{2mm} +\|u_0\|_{L^{2}_{y}L^{1+}_x}+\int^T_0\|uu_x\|_{H^{2+}_{xy}}dt+\int^T_0\|uu_x\|_{L^{2}_{y}L^{1+}_x}dt\notag \\
\lesssim& \|u_0\|_{Z_{s,(r_1, r_2)}}+	\|u\|^2_{X_T}.\label{dy52last}
\end{align}
The veracity of the last inequality follows readily from the same strategy as \eqref{NL2a1}. Collecting \eqref{nonlinddyyy}-\eqref{dy52last}, we deduce that $\big\|D_{y_j}^{\frac{5}{2}+}z_1\big\|_{L^2_{xy}}$ is bounded. 

Hence,  Theorem \ref{ZKnonlinSmooth} is proved.\hspace{89mm}$\square$

\subsection{Nonlinear smoothing for $k\geq2$}

\textbf{Proof of Theorem \ref{gZKDuhamelkg2}. } The arguments utilized to show this theorem  are
the local smoothing effect and maximal function estimates. We only deal with the case $s=2$, because  techniques we used here are applicable to larger $s$. Observe that $z_k\in L^2(\mathbb{R}^2)$, it suffices to control the $L^2$ norms of 
$$\partial_x^3\int_0^t W(t-t')u^k\partial_xu(t')dt' \hspace{5mm} \text{and}  \hspace{5mm} \partial_{y_j}^3\int_0^t W(t-t')u^k\partial_xu(t')dt'.$$ 

In view of ﻿the dual version of the smoothing effect \eqref{linear estimate:Kato-2} and H\"older's inequality, we get
\begin{align}
 &\left\|\partial_x^3\int_0^t W(t-t')u^k\partial_xu(t')dt' \right\|_{L^2_{xy}} 
\lesssim   \|\partial_x^2(u^k\partial_xu)\|_{L^1_xL^2_{yT}} \notag \\
\lesssim & \|u^ku_{xxx}\|_{L^1_xL^2_{yT}} + \|u^{k-1}u_{x}u_{xx}\|_{L^1_xL^2_{yT}}+ \|u^{k-2}u^3_{x}\|_{L^1_xL^2_{yT}}\notag \\
\lesssim & \|u\|^k_{L^{k}_xL^{\infty}_{yT}}\|u_{xxx}\|_{L^{\infty}_xL^2_{yT}} +  \|u\|^{k-1}_{L^{2(k-1)}_xL^{\infty}_{yT}}\|u_x\|_{L^{2}_xL^{\infty}_{yT}}\|u_{xx}\|_{L^{\infty}_xL^2_{yT}}+ \|u^{k-2}u^3_{x}\|_{L^1_xL^2_{yT}}\notag \\
=&T_1+T_2+T_3.\nonumber
\end{align}
 Firstly, let us consider $T_1$ and $T_2$. By using Sobolev embedding theorem, one has
$$\|u\|_{L^{k}_xL^{\infty}_{yT}}\lesssim \|J^{1/2-1/k}u\|_{L^{2}_xL^{\infty}_{yT}}$$
and 
$$\|u\|_{L^{2(k-1)}_xL^{\infty}_{yT}}\lesssim \|J^{1/2-1/2(k-1)}u\|_{L^{2}_xL^{\infty}_{yT}}.$$
It follows from the local well-posedness theory that 
$$\|u\|_{L^{k}_xL^{\infty}_{yT}}, \hspace{3mm} \|u\|_{L^{2(k-1)}_xL^{\infty}_{yT}}, \hspace{3mm}\|u_x\|_{L^{2}_xL^{\infty}_{yT}}, \hspace{3mm} \|u_{xx}\|_{L^{\infty}_xL^2_{yT}}, \hspace{3mm}\text{and}  \hspace{3mm}\|u_{xxx}\|_{L^{\infty}_xL^2_{yT}}$$
are bounded. Hence, we control $T_1$ and $T_2$.

Secondly, we consider $T_3$. If $k=2$, by H\"older's inequality, one obtains
$$T_3=\|u^3_{x}\|_{L^1_xL^2_{yT}}\lesssim \|u_x\|^2_{L^{2}_xL^{\infty}_{yT}}\|u_{x}\|_{L^{\infty}_xL^2_{yT}}$$
which is bounded via the local well-posedness theory. If $k\geq3$, we have
$$T_3=\|u^{k-2}u^3_{x}\|_{L^1_xL^2_{yT}}\lesssim \|u\|^{k-2}_{L^{2(k-2)}_xL^{\infty}_{yT}}\|u^3_{x}\|_{L^2_{xyT}}\lesssim \|u\|^{k-2}_{L^{2(k-2)}_xL^{\infty}_{yT}}\|u_{x}\|_{L_T^{2}L^{\infty}_{xy}}\|u_{x}\|^2_{L^{\infty}_TL^4_{xy}}$$
that is bounded by \eqref{linear estimate: strichartz cor1}, Sobolev embedding theorem and the local well-posedness theory.

Now we turn to estimate 
$$\partial_{y_j}^3\int_0^t W(t-t')u^k\partial_xu(t')dt'.$$ 
Using ﻿the dual version of the smoothing effect \eqref{linear estimate:Kato-2} again, we see that
\begin{align}
 &\left\|\partial_{y_j}^3\int_0^t W(t-t')u^k\partial_xu(t')dt' \right\|_{L^2_{xy}}
\lesssim  \|\partial_{y_j}^2(u^k\partial_xu)\|_{L^1_xL^2_{yT}} \notag \\
\lesssim & \|u^ku_{xy_jy_j}\|_{L^1_xL^2_{yT}} + \|u^{k-1}u_{xy_j}u_{y_j}\|_{L^1_xL^2_{yT}}+\|u^{k-1}u_{y_jy_j}u_{x}\|_{L^1_xL^2_{yT}}+ \|u^{k-2}u^2_{y_j}u_{x}\|_{L^1_xL^2_{yT}}\notag \\
\lesssim & \|u\|^k_{L^{k}_xL^{\infty}_{yT}}\|u_{xy_jy_j}\|_{L^{\infty}_xL^2_{yT}} +  \|u\|^{k-1}_{L^{2(k-1)}_xL^{\infty}_{yT}}\|u_{y_j}\|_{L^{2}_xL^{\infty}_{yT}}\|u_{xy_j}\|_{L^{\infty}_xL^2_{yT}}\notag \\
&\hspace{3mm}+\|u\|^{k-1}_{L^{2(k-1)}_xL^{\infty}_{yT}}\|u_{x}\|_{L^{2}_xL^{\infty}_{yT}}\|u_{y_jy_j}\|_{L^{\infty}_xL^2_{yT}}
+ \|u^{k-2}u^2_{y_j}u_{x}\|_{L^1_xL^2_{yT}}.\nonumber
\end{align}
By analogy with what has gone before, one can deduce that 
$$\partial_{y_j}^3\int_0^t W(t-t')u^k\partial_xu(t')dt'$$ 
is bounded. This completes the proof of the theorem. \hspace{56mm}$\square$

\section*{Acknowledgments}
 M.S is partially supported by the NSFC, Grant No. 12101629. The authors  would like to express their deep gratitude  to Professor Chenjie Fan from Academy of Mathematics and Systems Science, Chinese Academy of Sciences, for useful discussion on the construction of linear dispersive blow up solution.  Moreover, his helpful comments have improved the presentation of this paper. 

\end{document}